\documentclass[12pt,leqno]{article}
\usepackage{graphicx,psfrag,amssymb,amsmath,amsthm}
 \usepackage[dvips]{color}
  \usepackage{hyperref,psfrag,comment}
\usepackage{makeidx}
   \usepackage{eucal}
   \usepackage{graphicx,psfrag,amssymb,amsmath,amsthm, comment}
\usepackage{pdfsync}
\usepackage{hyperref}
\usepackage{color}
   \newtheorem{theorem}{Theorem}
\newtheorem{lemma}[theorem]{Lemma}

\newtheorem{proposition}[theorem]{Proposition}

\newtheorem{defn}[theorem]{Definition}

\def\beq{\begin{equation}}
\def\eeq{\end{equation}}

\def\dist{{\rm dist}}

\def\bB{\mathbf{B}}

\def\bH{\mathbf{H}}
\def\bI{\mathbf{I}}

\def\bg{\mathbf{g}}
\def\bh{\mathbf{h}}

\def\bs{\mathbf{s}}

\def\by{\mathbf{y}}

\def\cA{\mathcal{A}}
\def\cB{\mathcal{B}}
\def\cC{\mathcal{C}}
\def\cD{\mathcal{D}}

\def\cF{\mathcal{F}}

\def\cH{\mathcal{H}}

\def\cK{\mathcal{K}}

\def\cM{\mathcal{M}}

\def\cO{\mathcal{O}}

\def\cR{\mathcal{R}}
\def\cS{\mathcal{S}}
\def\cT{\mathcal{T}}
\def\cV{\mathcal{V}}
\def\cW{\mathcal{W}}

\def\tf{\tilde{f}}
\def\tg{\tilde{g}}

\def\dist{\text{\rm dist}}

\newcommand{\eps}{\varepsilon}

\def\pQ{\partial Q}

\numberwithin{equation}{section}

\begin{document}

\title{Decay of correlations for billiards\\
 with flat points II:
cusps effect}
\author{ Hong-Kun Zhang$^{1}$ }

\date{\today}

\maketitle
\footnotetext[1]{ Department of Mathematics \& Statistics, University of Massachusetts Amherst, MA 01003; Email: hongkun@math.umass.edu}

\begin{abstract}
 In this paper we continue to study billiards with flat points, by constructing a special	 family of  dispersing billiards with cusps. All boundaries of the table have  positive curvature except that the curvature vanishes at  the vertex of  cusps, i.e. the boundaries intersect at the flat point tangentially.
We study the mixing rates of this  one-parameter family of
billiards parameterized by $\beta\in (2,\infty)$, and  show that the
correlation functions of the collision map decay polynomially with
order $\cO(n^{-\frac{1}{\beta-1}})$ as $n\to\infty$. In particular, this solves an open question raised by Chernov and Markarian in  \cite{CM05}.

\end{abstract}

\tableofcontents
\section{Background and the main results.}

Dispersing billiards introduced by Sinai are uniformly hyperbolic
and have strong statistical properties. However, if the billiard table has
cusps or flat points, then its hyperbolicity
is nonuniform and statistical properties deteriorate.

Billiards with flat points	 were constructed and studied by Chernov and Zhang in \cite{CZ2, Z12, Z15}. It was proved that the mixing rates vary between $\cO(1/n)$ and exponentially fast depending on the parameter $\beta\geq 2$, as $n\to\infty$.
The main reason	 is that there exists one periodic trajectory between two flat points, which acts as a trap to slow down the mixing rates of nearby trajectories. In \cite{Z12,Z15},  statistical properties of a semidispersing billiard with flat points on the convex boundary were investigated. The decay rates were proven to be dominated by the so-called ``channel effect", which is essential for semidispersing billiards, with the existence of a pair of parallel trajectories tangential to all convex boundaries.

The first rigorous analysis of correlations
for dispersing billiards with cusps was given by Chernov and Markarian in \cite{CM05}, where they proved the rates of correlations (see the correlation function defined as in (\ref{Cn})) decay at $\cO((\ln n)^2/n)$, as $n\to\infty$. The rates were improved to $\cO(n^{-1})$ in \cite{CZ3}.  This model was further investigated in \cite{BCD1, BCD2}. In \cite{CM05}, Chernov and Makarian also raised an open question: ``It is interesting to let the curvature vanish at the vertex of the
cusp, ....  would this affect the rate of the
decay of correlations?"
Another interesting question to ask is that since all known billiards have decay rates at least of order $\cO(n^{-1})$, are there any  billiards with slower mixing rates?

 To answer both questions, in this paper we investigate the model proposed by Chernov and Markarian in \cite{CM05}, and show that this family of billiards enjoys arbitrarily slower decay rates.    More precisely, we first take MachtaÕs three-arc table $Q_1$  as studied by Chernov and Markarian \cite{CM05}, with boundary consisting of $3$  smooth curves $\Gamma_i'$, $i=1\cdots, 3$; and  $Q_1$ has three cusps at the intersection points. Then we smoothly deform these curves at their end points, and denote the new curves as $\Gamma_i$; such that they all have zero  derivatives up to $\beta-1$ order  at  end points, and the $\beta$-th order derivative is not zero, for $\beta\in (2,\infty)$.  We let $Q=Q_{\beta}$ be  bounded by $\Gamma_1, \Gamma_2,\Gamma_3$, for $\beta\in (2,\infty)$.  Indeed according to our above assumption,  if we choose a Cartesian coordinate system $(s,z)$ with origin at any of these cusp point, denoted as $P$, with the horizontal $s$-axis being the tangent line to the boundary of the billiard table, then   the billiard table satisfies the following three conditions:
\begin{itemize}
\item[\textbf{(h1)}]  We assume  for some small $\eps_0>0$,  the pair of boundary adjacent to $P$ can be represented in the $\eps_0-$neighborhood of the cusp $P$ as:
 \beq\label{z1s}
 z_1(s)=\beta^{-1} s^{\beta},\,\,\,\,\,\,z_2(s)=-\beta^{-1} s^{\beta},\,\,\,\,\,\,\forall s\in [0,\eps_0]\eeq
     \item[\textbf{(h2)}]  We also assume  that the tangent line of the table at the cusp  $P$ will hit the opposite boundary at a point, called $D$,  perpendicularly.
  \end{itemize}

   We will investigate the statistical properties of the billiard system on $Q_{\beta}$.  The billiard	 flow $\Phi^t$ is defined on the unit sphere bundle $Q\times	 \mathbf{S}^1$ and preserves the
Liouville measure.
This type of billiards  can be viewed as a special type of semi-dispersing billiards, as its boundary contains points with zero curvature. Semi-dispersing billiards have been proven to	 enjoy strong ergodic
properties: their continuous time dynamics and the billiard ball maps
are both completely hyperbolic, ergodic, K-mixing and Bernoulli, see
\cite{CH96, GO, OW98, Sin70,SC87,CM} and the references therein. However, these systems	
have quite different statistical properties depending on the geometric properties of the billiard table. Figure \ref{fig1}.
describes  a billiard table with cusp at the flat point $P$ for $\beta>2$.
\begin{figure}[h]
\centering \psfrag{Q}{\scriptsize$Q$} \psfrag{z}{\scriptsize$z$}
 \psfrag{s}{\scriptsize$s$}\psfrag{1}{\scriptsize$\bB$}
\includegraphics[width=3in]{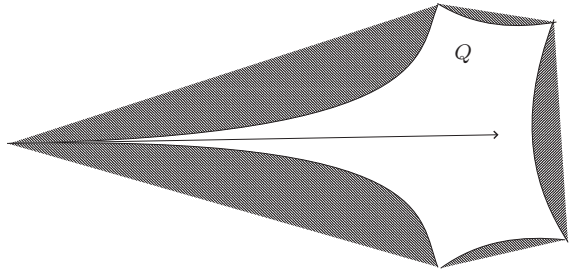}
\renewcommand{\figurename}{Fig.}
\caption{\small\small\small\small\small{A	 table with a cusp at a flat point for $\beta\in (2, \infty)$}}\label{fig1}
\end{figure}

There is a natural cross section $\cM$ in	 $Q\times	 \mathbf{S}^1$ that contains all postcollision vectors based at the boundary of the table $\partial Q$. The set $\cM=\partial Q \times [-\pi/2, \pi/2]$  is called the collision space.  Any postcollision vector $x\in \cM$ 	 can be represented by $x=(r, \varphi)$, where $r$ is the
arclength parameter along $\partial Q$,	starting from an end point of $\partial Q$, measured in the clockwise direction; and $\varphi\in [-\pi/2, \pi/2]$ is the angle that $x$	 makes with the inward unit normal vector to the boundary.

The corresponding Poincar\'{e} map (or the billiard map) $\cF: \cM\to \cM$ generated by the
collisions of the particle with $\pQ$ preserves a natural
absolutely continuous measure $\mu$ on the collision space $\cM$, such that
$$d\mu=\frac{1}{2|\partial Q|}\, \cos\varphi dr\,d\varphi.$$

    For any square-integrable observable  $f,g\in L^2_{\mu}(\cM)$,  {\em
correlations} of $f$ and $g$ are defined
by \beq
   \cC_n(f,g,\cF,\mu) = \int_{\cM} (f\circ \cF^n)\, g\, d\mu -
	\int_{\cM} f\, d\mu    \int_{\cM} g\, d\mu.
	   \label{Cn}
\eeq

The mixing speed of the system $(\cM,\cF,\mu)$ is characterized by
the \emph{rate of decay of correlations}, i.e., by the speed of
convergence to $0$ of (\ref{Cn}) for ``good enough'' functions $f$ and $g$.

Let $S_{\pm n}$ be the singular set of the map $\cF^{\pm n}$, for any $n\geq 1$, and $S_{n_1, n_2}:=\cup_{m=n_1}^{n_2} S_{m}$ the union of these singular sets, for any integers $n_1, n_2$.   For any $\gamma\in (0,1)$, let $\cH(\gamma)$	be the set of all  bounded real-valued  functions $f\in L_{\infty}(\cM,\mu)$, such that  there exist integer $n_1,n_2$, for any connected component $A\in \cM\setminus S_{n_1, n_2}$, any $x, y\in A$,
\beq \label{DHC-} 	|f(x) - f(y)| \leq \|f\|_{\gamma} \dist(x,y)^{\gamma},\eeq
with
$$\|f\|_{\gamma}\colon = \sup_{A\in S_{n_1,n_2}}\sup_{ x, y\in A}\frac{|f(x)-f(y)|}{\dist(x,y)^{\gamma}}<\infty.$$

For every $f\in \cH(\gamma)$ we define
\beq \label{defCgamma}
   \|f\|_{C^{\gamma}}\colon=\|f\|_{\infty}+\|f\|_{\gamma}.
\eeq

In this paper we  obtain the following  results.
\begin{theorem} \label{TmMain}
For the family of  billiards on $Q_{\beta}$ defined as (\textbf{(h1)}-\textbf{(h2)}), with  $\beta>2$,  Then for any  $\gamma\in (0,1]$, any
observables $f,g\in \cH(\gamma)$  on $\cM$, there exists $C_{f,g}=C(f,g)>0$, such that
$$|\mu(f\circ \cF^n\cdot g)-\mu(f)\mu(g)|\leq  C_{f,g} n^{-\frac{1}{\beta-1}},
$$ for $n\geq 1$.
\end{theorem}

For the case when $\beta=2$, the system  corresponds to the dispersing billiards with cusps  and enjoys  	 mixing rates of order $\cO(n^{-1})$, see \cite{CM05,CZ3}.

  \noindent Convention. We use the following notation: $A\sim B$ means that $C^{-1}\leq A/B\leq C$ for some constant $C>1$ . Also, $A=\cO(B)$  means that
$|A|/B<C$  for some constant $C>0$. From now on, we will denote by $C>0$ various constants (depending only on the table) whose exact
values are not important.

\section{General scheme}\label{Sec:3}

 Based upon the methods by Young \cite{Y98},   a general scheme was developed by Markarian, Chernov and Zhang
\cite{M04, CZ, CZ2, CZ3, CZ09}  on obtaining slow rates of hyperbolic systems with singularities and applied the method on different models. Let $M\subset \cM$ be a nice subset, such that the induced map $F : M\rightarrow M $
 is strongly (uniformly) hyperbolic. One can easily check that it preserves the measure
$\mu_M$ obtained by conditioning $\mu$ on $M$.  For our billiards, hyperbolicity deteriorates only as the
moving particle gets trapped by a cusp, when it experiences a large number of
rapid collisions near the corner point of the cusp.

 In this paper, we first fix a number $K_0>1$, and  call any sequence of successive collisions
of length $>K_0$ in a cusp a {\it{corner series}}.  In particular, we post an upper bound for $K_0$,  such that for  any grazing collision $x\in \cM$,  its forward trajectory will enter a corner series of length $> K_0$. We thus define
\beq\label{defM}
M=\{x\in \cM\,:\,   \text{forward successive collisions of $x$ in any cusp has
 length} \leq K_0\}.
 \eeq
  Clearly, there exists $\varphi_{K_0}\in (0,\pi/2)$, such that \beq\label{MphiKo}
 |\varphi|\leq \varphi_{K_0},\,\,\,\,\,\forall \varphi\in M,
 \eeq
 i.e. $M\subset \cM$ stays away from $\varphi=\pm \pi/2$.
 For any $x\in M$
we call $$R(x) = \min\{n\geq  1: \cF^n(x)\in M\}$$ the return time function and thus the  return map $F\colon M \to M$ is defined by
\beq \label{Fdef}
   F(x) = \cF^{R(x)}(x),\,\,\,\,\,\forall x\in M.
\eeq

In order to prove Theorem~\ref{TmMain},  the
the strategy  consists of two steps;  they were fully described in \cite{CZ,CZ3}, and has been applied to several
classes of billiards with slow mixing rates, see  \cite{CZ2,CM05}, so we will not bring
up unnecessary details here.\\

\noindent \textbf{(F1)} First, the map $F\colon M\to M$ enjoys
exponential decay of correlations.

\noindent{\small{ More precisely, for any Holder observables $f,g\in\cH(\gamma)$ on $M$ with H\"older exponent $\gamma\in (0,1)$,
$$|\int_{M} (f\circ F^n)\, g\, d\mu_M -
    \int_{M} f\, d\mu_M    \int_{M} g\, d\mu_M|\leq C \|f\|_{C^{\gamma}}\|g\|_{C^{\gamma}} \vartheta^n,$$
for some uniform constant $\vartheta=\vartheta(\gamma)\in (0,1)$ and $C>0$. }}\medskip

\noindent \textbf{(F2)} Second, the distribution of the return time function $R: M\to [1, \infty)$ satisfies:
  $$\mu_M( R\geq n)\sim \frac{1}{n^{1+a}},$$
for  some $a>0$ and any large $n $.

 \medskip

\noindent It was proved in \cite{CZ} that assumption \textbf{(F1)-(F2)} imply polynomial decay rates of order  $\cO(n^{-a}(\ln n)^{1+a})$. Also, the proof of (\textbf{F1}) is reduced in
\cite{CZ} to the verification of a one-step expansion condition, as well as the regularities of the invariant manifolds for the induced map. To improve the upper bound for the decay rates, one needs to analyze the statistical properties of the return time function.  In \cite{CZ3}, the upper bound of decay rates of correlations was improved by dropping the logarithmic factor.  Mainly because points in the region $(R>n)$  has tendency to move to $(R<n)$ fast enough under further iterations of the billiard map.

The paper is organized as following.  In  Section 3, we investigate the asymptotic quantities for any typical, long corner series. In Section 4, we construct the induced system $(F,M)$, by removing  those corner series. The hyperbolicity of the reduced map is also proved in Section 4. The assumption \textbf{(F2)} is verified in Section 5, by analyzing the distribution of the return time function. The exponential decay of correlations for the reduced map and \textbf{(F1)} was proved in Section 6, by verifying the One-step expansion estimates, see Lemma \ref{onestep}.    In Section 7, we get the improved upper bound using method as in \cite{CZ3}.

\section{The corner series}
In this section,  we first investigate the geometry of corner series, which correspond to  certain billiard trajectories
entering the cusp and  experiencing a large number of refections there before
getting out.  To simplify our analysis we consider here the cusp made by $\Gamma_1, \Gamma_1$ with a common tangent line at the end point $P$.   For any $N\geq K_0$, we define $M_N$ to be  the set of points in $M$ whose forward trajectories go down  the cusp at $P$ for a corner series of length $N$.
For simplicity,  \textbf{we assume the flat point $P$ has $r$-coordinate $r_f. $}

Let $N>K_0$ be the number of reflections in the corner series  starting from a vector $x\in M_N$. For any $x=(r,\varphi)\in M_N$, let  $x_N=F x$, and $x_n:=\cF^{n} x=(r_n, \varphi_n)$, for $n=1,\cdots, N-1$, denote the set of all points of reflection on $\Gamma_1'\cup\Gamma_2'$.  We also call $\{x_n, n=1,\cdots, N-1\}$ a corner series of length $N$ generated by $x\in M_N$.

We choose a Cartesian coordinate system $(s,z)$ with origin at $P$, the horizontal $s$-axis being the tangent line to the boundary of the billiard table. By \textbf{(h1)},     for some small $\eps_0>0$,  the pair of boundary adjacent to $P$ can be represented in the $\eps_0-$neighborhood of the cusp $P$ as:
 \beq\label{zsP}
 z(s)=\pm \beta^{-1} s^{\beta},\,\,\,\,\,\,\forall s\in [0,\eps_0]\eeq

We denote $s_n$ to be the $s$-coordinate of the base point of $x_n=(r_n,\varphi_n)$. By the smoothness of the boundary curves,
\beq\label{rnsn}|r_n-r_f|=\int_0^{s_n} \sqrt{1+|z'(s)|^2} \,ds=s_n+\cO(s_n^{\beta}).\eeq

To estimate the tail distribution of $\mu_M(R\geq n)$, we will fix a large number $N_0$, and only consider those corner series, such that $N>N_0$. We will also work
with more convenient coordinates:
$$\gamma_n=\pi/2-|\varphi_n|, \,\,\,\,\,\,\,\text{ and }\,\,\,\,\,\,\,\alpha_n=\tan^{-1}(s_n^{\beta-1}).$$

Note that by  (\ref{zsP}),  the tangent vector of $\partial Q$ at $s_n$ is $(1, s_n^{\beta-1})$, which implies that
\beq \label{alphansn}\alpha_n=\tan^{-1}(s_n^{\beta-1})=s_n^{\beta-1}+\cO(s_n^{3(\beta-1)})\eeq stands for  the angle of the tangent vector  at $s_n$ made with the horizontal axis, or equivalently, with the tangent line through the flat point $P$.  Note  that both $\alpha_n$ and
$\gamma_n$ are positive for $1\leq n\leq N-1$; $\alpha_n$ are all small  if $N$ is  large enough. While $\gamma_n$
are initially small, they slowly grow to
about $\pi/2$ for $n\sim N/2$, and then again decrease and
get small.   We use notations similar to that of \cite{CM05}, and define
$$\alpha_{\bar N}:=\min\{\alpha_n\,:\, 1\leq n\leq N\}.$$
It was proved in \cite{CM05}~ Lemma 3.1, that $\bar N$ is almost the middle point of $N$, i.e. $|\bar N-N/2|\leq 2$, if $\beta=2$. The proof only relies on the symmetry of the boundary of the billiard table near the cusp, so it also apply to our case words by words, even for $\beta>2$.
We further subdivide the corner series into three segments. We fix a small enough
$\bar\gamma$ and  let $$N_1=\max\{ n\leq \bar N\,:\, \gamma_n<\bar\gamma\},\,\,\,\, N_3=\max\{n\geq \bar N\,:\, \gamma_n>\bar\gamma\}.$$ And put $N_2=\bar N$.
We call the segment on $[1, N_1]$ the ``entering period" in the corner series, the
segment $[N_1 + 1;N_3-1]$ the ``turning period", and the segment $[N_3,N]$
its ``exiting period". Clearly $|N_3+N_1-N|\leq 2$.

By the symmetry of the billiard table, it is enough to consider the first half of the series, $1\leq n\leq N_1$.

Using these relations, one has the following proposition for a corner series of length $N$ generated by any $x\in M_N$.
\begin{proposition}\label{prop3} The following are true:\\
(1)  $N_1\sim N_2-N_1\sim N_3-N_2\sim N-N_3\sim N.$ i.e. all  three segments in the corner series have length of order $N$;\\
(2) $\alpha_1\sim N^{-\frac{\beta}{2\beta-1}},\,\,\,\, \alpha_{n}\sim n^{-1}\sim N^{-1}$, for $n\in [N_1, N_2]$;\\
(3) $\alpha_n\sim (n^{\beta-1}N^{\beta})^{-\frac{1}{2\beta-1}}$, for $n\in [1,N_1]$;\\
(4) $\gamma_1=\cO(N^{-\frac{\beta}{2\beta-1}}), \,\,\gamma_2\sim N^{-\frac{\beta}{2\beta-1}}$;\\
(5) $\gamma_n\sim (n N^{-1})^{\frac{\beta}{2\beta-1}}$, for $n\in [1,N_1]$;\\
(6) For  $N$ sufficiently large, the  quantity $\{H_{N}((r_n,\varphi_n)), n=1,\cdots, N-1\}$ is almost invariant along a corner series of length $N$:
$$H_{N}((r_n,\varphi_n))=|r_n-r_f|^{\beta} \cos \varphi_n=C_N+\cO(N^{-\frac{2\beta-1}{\beta-1}}),$$
for any $n=N_1,\cdots, N_2$, with $C_N=C N^{-\frac{\beta}{\beta-1} }$, for some uniform constant $C>0$. For $n=1,\cdots, N_1$, we have
$$H_N((r_n,\varphi_n))=|r_n-r_f|^{\beta} \cos \varphi_n=C_N+\cO(n^{-1}N^{-\frac{\beta}{\beta-1}}).$$

\end{proposition}
\begin{proof}
 By the symmetric property of $\Gamma_1'$ and $\Gamma_2'$, we will now only concentrate on $n=1,\cdots, N_2\sim N/2$. Note  that both $\alpha_n$ and
$\gamma_n$ are positive for $1\leq n\leq N_2$. $\{\alpha_n\}$ is a decreasing sequence, and  are all small; while $\gamma_n$ is an increasing sequence, which
is initially small, the terms slowly grow to
about $\pi/2$ for $n\sim N/2$.

The following
equations are simple geometric facts:
\beq\label{3.61}
\gamma_{n+1}=\gamma_{n}+\alpha_n+\alpha_{n+1},\eeq
\beq\label{3.71}
s_{n+1}=s_n-\tau_n \cos(\gamma_n+\alpha_n),
\eeq
where we denote
\beq\label{taun}
\tau_n=\frac{1}{\beta}\cdot \frac{s_n^{\beta}+s_{n+1}^{\beta}}{\sin(\gamma_n+\alpha_n)}.\eeq
 Here $\tau_n$ is the free path between two collisions based at $s_n$ and $s_{n+1}$.

Now we denote $v_n=\gamma_n+\alpha_n$, then
\beq\label{3.6}
v_{n+1}=v_{n}+2 \alpha_{n+1}.\eeq
In addition (\ref{3.71}) can also be written as
\beq\label{3.7}
s_{n+1}=s_n-  \frac{s_{n+1}^{\beta}+s_n^{\beta}}{ \beta \tan v_n}.
\eeq
Using the mean-value theorem, we know for any $d>1$, and $n$ large,  there exists $s_{n,d,*}\in [s_{n+1}, s_n]$, such that \beq\label{snpower}
s_{n+1}^{d}-s_n^{d}=d \,s_{n,d,*}^{d-1}(s_{n+1}-s_n)=- d \,s_{n,d,*}^{d-1} \frac{s_{n+1}^{\beta}+s_n^{\beta}}{ \beta \tan v_n}.
\eeq

Combining with (\ref{alphansn}), we know that by the mean-value theorem, there exists $s_{n,*}\in [s_n, s_{n+1}]$, such that
\begin{align*}
\alpha_{n+1}-\alpha_n&=\tan^{-1}(s_{n+1}^{\beta-1})- \tan^{-1}(s_{n+1}^{\beta-1})=\frac{s_{n+1}^{\beta-1}-s_n^{\beta-1}}{1+s^{2(\beta-1)}_{n,*}}\\
&=(s_{n+1}^{\beta-1}-s_n^{\beta-1})(1-s^{2(\beta-1)}_{n,*}+\cO(s^{4(\beta-1)}_{n}))\\
&=- (\beta-1) \,s_{n,\beta-1,*}^{\beta-2} \frac{s_{n+1}^{\beta}+s_n^{\beta}}{ \beta \tan v_n}(1-s^{2(\beta-1)}_{n}+\cO(s^{3(\beta-1)}_{n}/\tan v_n))\end{align*}
where we used (\ref{snpower}) for $d=\beta-1$ in the last step. Using  (\ref{snpower}) one more time, we get
\begin{align*}
s_{n+1}^{\beta}s_{n,\beta-1,*}^{\beta-2}+s_n^{\beta}s_{n,\beta-1,*}^{\beta-2}&=s_{n+1}^{2\beta-2}+s_n^{2\beta-2}+s_{n+1}^{\beta}(s_{n,\beta-1,*}^{\beta-2}-s_{n+1}^{\beta-2})+s_{n}^{\beta}(s_{n,\beta-1,*}^{\beta-2}-s_{n}^{\beta-2})\\
&=s_{n+1}^{2\beta-2}+s_n^{2\beta-2}+\cO(s^{3(\beta-1)}_{n}/\tan v_n))\end{align*}
Combining above facts, as well as (\ref{alphansn}), we get
\beq\label{3.7'}
\alpha_{n+1}-\alpha_n= -  \bar a\cdot \frac{\alpha_{n+1}^2+\alpha_n^2}{2  \tan v_n}+\cO(\alpha_n^3(\tan v_n)^{-2})=\alpha_n-  \bar a\cdot \frac{\alpha_n^2}{  \tan v_n}+\cO(\alpha_n^3(\tan v_n)^{-2}),\eeq
where $\bar a=2(\beta-1)/\beta$.

By (\ref{3.6}) and (\ref{3.7}), we have
\beq\label{3.8}
v_2>2\alpha_2,\,\,\,\,\sum_{n=1}^{N_2}\alpha_n\leq v_{N_2}/2\leq \pi/4.
\eeq
(\ref{3.7'}) implies that
$$ \bar a\cdot \frac{\alpha_n}{  \tan v_n}-1=-\frac{\alpha_{n+1}}{\alpha_n}+\cO(\alpha_n^2(\tan v_n)^{-2})<1.$$
Thus \beq\label{3.8'}\frac{\alpha_n}{\tan v_n}<1/\bar a.\eeq


We denote $A_n= s_n^{\beta}\sin v_n$, for $n=1,\cdots, N_2$. We would like to show that $A_n\sim A_{N_2}$ is almost invariant,  as long as $N$ is large.  To see this, using (\ref{3.6})-(\ref{3.7}), we get
 \begin{align}\label{en}e_n&:=A_{n+1}-A_n=s_{n+1}^{\beta}\sin v_{n+1}-s_n^{\beta}\sin v_n\nonumber\\
 &=\left(s_{n+1}^{\beta}-s_n^{\beta}\right)\sin v_{n+1}+s_{n}^{\beta}\left(\sin v_{n+1}-\sin v_n\right)\nonumber\\
 &=\beta s_{n,*}^{\beta-1}(s_{n+1}-s_n)\sin v_{n+1}+s_n^{\beta}\cos v_n^*(v_{n+1}-v_n)\nonumber\\
 &=-\frac{s_{n,*}^{\beta-1}(s_n^{\beta}+s_{n+1}^{\beta})\sin v_{n+1}}{\tan v_n}+2s_{n}^{\beta}s_{n+1}^{\beta-1}\cos v_n^*\nonumber\\
 &=-2s_{n}^{2\beta-1}\cos v_n + 2s_{n}^{2\beta-1}\cos v_n^*+\cO(s_n^{3\beta-2}\cos v_n)\nonumber\\
 &=4s_{n}^{3\beta-2}\sin v_n+\cO(s_n^{3\beta-2}\cos v_n), \end{align}
 where $s_{n,*}\in [s_{n+1},s_n]$ and $v_n^*\in [v_n, v_{n+1}]$.

 Note that for $n\in [N_1, N_2]$, $\bar\gamma < v_n\leq \pi/2$. This above estimation implies that
 $$s_n^{\beta}\sin v_n=s_{N_2}^{\beta}\sin v_{N_2}+\cO(\sum_{k=n}^{N_2}s_k^{3\beta-2} ),$$
 which implies that the error term  satisfies $$\sum_{k=n}^{N_2}s_k^{3\beta-2}\leq s_n^{\beta}.$$

Using (\ref{3.7}), we know that
\begin{align*}
s_{n+1}^{2\beta-1}-s_n^{2\beta-1}=-\frac{2\beta-1}{\beta}\frac{s_{n+1}^{3\beta-2}}{ \tan v_n}+\cO\left(\frac{s_n^{4\beta-3}}{\tan v_n}\right).\end{align*}
If we sum over for $k\in [n, N_2]$, this implies that
$$s_{n}^{2\beta-1}=s_{N_2}^{2\beta-1}+\frac{2\beta-1}{\beta}\sum_{k=n}^{N_2}\frac{s_{k+1}^{3\beta-2}}{ \tan v_k}+\cO\left(\sum_{k=n}^{N_2}\frac{s_k^{4\beta-3}}{\tan v_k}\right).$$
Again using the fact that $v_n<\pi/2$, for $n\in [1,  N_2]$, then we get
$$ \sum_{k=n}^{N_2}s_{k+1}^{3\beta-2}\cos v_k\leq \frac{2\beta}{2\beta-1} s_{n}^{2\beta-1}<s_{n}^{2\beta-1}.$$

Next we fix another very small number $\bar\gamma_1$, and $\tilde N_2<N_2$, such that $\tilde N_2=\max\{n<N_2\,:v_{n}<\pi/2-\bar\gamma_1\}$. Then for any $n\in [1,  \tilde N_2]$, then we get
$$ \sin\bar \gamma_1\sum_{k=n}^{\tilde N_2}s_{k+1}^{3\beta-2}\leq  \sum_{k=n}^{\tilde N_2}s_{k+1}^{3\beta-2}\cos v_k <s_{n}^{2\beta-1}.$$

Combining the above estimations together with the expression for $e_n$, we know that
$$\sum_{k-n}^{\tilde N_2} |e_k|\leq 4s_{n}^{2\beta-1}/\sin\bar\gamma_1.$$
This suggests that for $N$ large enough, for any $n=1,\cdots, \tilde N_2$, the quantity
$$s_n^{\beta} \sin v_{n} = s_{\tilde N_2}^{\beta}\sin v_{\tilde N_2}+\cO(s_n^{2\beta-1}).  $$ should be almost invariant. We define \beq\label{defCN}
C_N:=s_{\tilde N_2}^{\beta}\sin v_{\tilde N_2}=s_{\tilde N_2}^{\beta}\cos \bar\gamma_1,\eeq and \beq\label{defHn} A_n:=s_n^{\beta} \sin v_{n}.\eeq Then we have shown
\beq\label{Hn}
A_n=C_N(1+\cO(C_N^{-1}s_n^{2\beta-1}))\sim C_N,
\eeq
for any  $n\in [1, \tilde N_2]$.

Moreover for any $n\in [N_1, \tilde N_2]$, using the fact that $\bar\gamma<v_n<\pi/2-\bar\gamma_1$, thus
\beq\label{alphaN1N2}
s_n\sim s_{\tilde N_2},\,\,\,\,\,\forall n\in [N_1, \tilde N_2]
\eeq

Below, we will only concentrate on $n\in [1,N_1]$. Note (\ref{Hn}) implies that for $n\in [1,N_1]$, we have
\beq\label{Hn11}s_n^{\beta}\sin v_n = C_N^{-1}+\cO(s_n^{2\beta-1}),\,\,\,\,\,\,\text{ and }\,\,\,\,\,\sin v_n= \cos\bar\gamma_1\cdot \frac{s_{\tilde N_2}^{\beta}}{s_{n}^{\beta}}+\cO(C_Ns_n^{\beta-1})\sim \frac{s_{\tilde N_2}^{\beta}}{s_{n}^{\beta}}. \eeq

We now define \beq\label{uwn}
u_n=\frac{\alpha_n}{\alpha_{n+1}},\,\,\,\,\,w_n=\frac{v_n}{\alpha_{n}}
\eeq
It is important to find the asymptotic for $w_n$ and $u_n$. First note that equation (\ref{3.6}) yields
\beq\label{3.12}
w_{n+1}=w_n u_n+2=w_n+2+w_n(u_n-1).\eeq
 Since $u_n-1\geq 0$ and $w_n\geq 2$ by (\ref{3.8}), thus
 \beq\label{wnn}w_n\geq 2n,\eeq for any $1\leq n\leq N_1$.

Thus for $n\in [1,N_1]$, we have
\begin{align*}\label{un-1b}
 u_n-1&=\frac{\alpha_n-\alpha_{n+1}}{\alpha_{n+1}}= \frac{ \bar a(\alpha_{n+1}+u_n\alpha_n)}{2\tan v_n} +\cO(\alpha_n^2)=\cO(\frac{\alpha_n}{\tan v_n})<C,\end{align*}for some uniform constant $C>0$.
Combining with (\ref{3.8}), it  implies that for $1\leq n\leq N_1$, $0<u_n-1< C$ is uniformly bounded.

 Moreover,
 (\ref{3.7}) implies that
\beq\label{3.14} 1-u_n^{-1}=\frac{v_n}{\tan v_n}\cdot \frac{\bar a(u_n^2+1)}{2 w_n} +\cO(\alpha_n^2)\leq \frac{\bar a}{ w_n}+\cO(\alpha_n^2)\leq \frac{\bar a}{ 2 n}+\cO(\alpha_n^2),\eeq
and
$$\alpha_{n+1}\geq \frac{\alpha_n}{1+w_n^{-1}}+\cO(\alpha_n^3)\geq \alpha_n(1-\frac{1}{2n})+\cO(\alpha_n^3)+\cO(\alpha_n n^{-2}).$$
This implies that
$$\alpha_{n}\geq \frac{\alpha_1}{2} \prod_{m=2}^n (1-\frac{1}{m}) \geq \frac{\alpha_1}{cn}$$
for some constant $c>0$.

 We denote $u_n=1+b_n$, with $b_n\leq \frac{\bar a}{n}$, then one can check that
$$u_n^{-1}+u_n=2+\frac{b_n^2}{u_n}.$$
(\ref{3.14}) implies that
 $$u_n-1=\frac{v_n}{\tan v_n}\cdot \frac{\bar a( u_n^{-1}+u_n)}{2 w_n} +\cO(\alpha_n^2)=\frac{v_n}{\tan v_n}\cdot\frac{ \bar a}{ w_n} +\cO(\alpha_n^2).$$

Now(\ref{3.12}) implies that
$$w_{n+1}=w_n u_n+2\leq 2+w_n(1+b_n).$$

Combining (\ref{3.14}) with (\ref{3.12}) gives
\begin{align}\label{3.15}
w_{n+1}&=w_n+ 2+\bar a +\cO(n\alpha_n^2)+\cO(n^{-2})+\cO(v_n^2).
\end{align}
By (\ref{Hn11}), we know that
$$\sin v_n \sim \frac{\alpha_{N_1}^{\frac{\beta}{\beta-1}}}{\alpha_n^{\frac{\beta}{\beta-1}}}=\prod_{k=n}^{N_1} u_k^{-^{\frac{\beta}{\beta-1}}}\sim (n/N_1)^{\frac{\beta}{\beta-1}}.$$ Thus
$$v_n\sim (n/N_1)^2.$$
Let $\Gamma_n=\sum_{i=1}^n v_i^2$, one can now show that   $\sum_{i=1}^n \Gamma_i/i^2$ is bounded:
$$\sum_{i=1}^n\Gamma_i/i^2\leq C\sum_{i=1}^n i^3/N_1^4\leq C,$$
for any $n\in [1,N_1]$.
The lower bound in (\ref{wnn}) now implies
\begin{align}\label{wnEn}
w_n&=w_1+(2+\bar a)n  +E_n,
\end{align}
where $E_n\leq C_1\ln n+C_2 \Gamma_n$,
for some constant $C_1, C_2>0$.

Now we use (\ref{wnEn}) to estimate $u_n$. Note that
 \begin{align}\label{anan+1}
 1-u_n^{-1}&=\frac{v_n}{\tan v_n}\cdot \frac{\bar a(u_n^2+1)}{2w_n} +\cO(\alpha_n^2)\nonumber\\
 &=\frac{\bar a}{(2+\bar a)n}\cdot \frac{v_n}{\tan v_n}\cdot \frac{1}{1+\frac{\bar a( w_1+C_1\ln n+C_2\Gamma_n)}{3n}} +\cO(\alpha_n^2)\nonumber\\
 &= \frac{\bar a}{(2+\bar a)n +E_{2,n}}  +\cO(\alpha_n^2),
 \end{align}
 where $0\leq E_{2,n}<C_3\ln n+C_4\Gamma_n$ for some uniform constants $C_i>0$, $i=1,2,3$. Using the definition of $u_n$, we have obtain
 $$1-u_n^{-1}=\frac{\alpha_{n+1}-\alpha_n}{\alpha_n}=\frac{\bar a}{(2+\bar a)n}+\cO(\frac{1}{n^2\ln n}).$$
 Combining with (\ref{3.7'}), we get
 \beq\label{antanvn}
\frac{\alpha_n}{\tan \gamma_n}=\frac{1}{(2+\bar a)n}+\cO(\frac{1}{n^2\ln n}).
 \eeq

 This implies that
 $$\frac{\bar a}{(2+\bar a)n }  +\cO(\alpha_n^2)\leq u_n-1\leq \frac{\bar a}{(2+\bar a)n +C_3\ln n+C_4 \Gamma_n}  +\cO(\alpha_n^2).$$
Using the definition $\alpha_{n+1}=\alpha_n/u_n$, we get
$$\alpha_n=\alpha_1 \exp\left(-\sum_{i=1}^n \ln (1+ \frac{\bar a}{(2+\bar a)n +E_{2,n}})\right) \sim \alpha_1 n^{-\bar a/(2+\bar a)}.$$
Thus implies that  for $n\in [1,N_1]$,
\beq\label{alphan}
\alpha_n\sim \alpha_1 n^{-\bar a/(2+\bar a)}\sim \alpha_1 n^{-\frac{\beta-1}{2\beta-1}}.
\eeq
Combining with (\ref{wnEn}), we have
$$w_n=(w_1+(2+\bar a)n) \alpha_n \sim (2+\bar a)  \alpha_1 n^{\frac{2}{2+\bar a}}\sim \alpha_1 n^{\frac{\beta}{2\beta-1}}.$$
Since $v_{n}\sim \bar \gamma$, for $n\in[N_1,N_2]$, so $$\alpha_1\sim N_2^{-\frac{2}{2+\bar a}}\sim N_1^{-\frac{2}{2+\bar a}}\sim N^{-\frac{\beta}{2\beta-1}}.$$Thus we have $N_1\sim N_2\sim N$.

This further implies that for $n\in [2,N_2]$, we have
\beq\label{N1N2}\alpha_n\sim \frac{1}{(n^{\beta-1}N^{\beta})^{\frac{1}{2\beta-1}}},\,\,\,\,\,\gamma_n\sim v_n\sim (n N^{-1})^{\frac{\beta}{2\beta-1}}.\eeq
In particularly, for $n\in [N_1,N_2]$, using the above fact that $N_1\sim N_2\sim N$, we have
\beq\label{N1N22}\alpha_n\sim n^{-1}\sim N^{-1},\,\,\,\,\,v_n\sim \bar\gamma.
\eeq

Now we move back to improve the estimation of $A_n$.
 Note that $$C_N= \sin v_{\tilde N_2} s_{\tilde N_2}^{\beta} =\cos\bar\gamma_1 s_{\tilde N_2}^{\beta}\sim N^{-\frac{\beta}{\beta-1}}.$$
Now by (\ref{Hn}), we know that
$$A_n=s_n^{\beta}\sin v_n=C_N+\cO(s_n^{2\beta-1})),$$
for any $n=1,\cdots, \tilde N_2$.

On the other hand, (\ref{en}) implies that for $n=1,\cdots, N_2$,
$$A_n-A_{n+1}=\cO(s_{n}^{3\beta-2}\sin v_n)+\cO(s_n^{3\beta-2}\cos v_n).$$
Thus for $n\in [1,N_1]$, we have
$$A_n=A_{N_1}+\cO(\sum_{k=n}^{N_1}s_{n}^{3\beta-2}\sin v_n)+\cO(\sum_{k=n}^{N_1}s_n^{3\beta-2}\cos v_n)=A_N+\cO(n^{-1}N^{-\frac{\beta}{\beta-1}}),$$
where we have used (\ref{N1N2}) in the last step. Moreover, using (\ref{N1N22}), we also get
$$A_n=s_n^{\beta}\sin v_n=C_N+\cO( N^{-\frac{2\beta-1}{\beta-1}})),$$
for any $n=N_1,\cdots,  N_2$.
Since $v_n=\gamma_n+\alpha_n$, we will now define a new quantity:
\beq\label{defnHn}
H_{N}((r_n,\varphi_n))=|r_n-r_f|^{\beta}\cos \varphi_n= s_n^{\beta}\sin\gamma_n+\cO(s_n^{2\beta}\sin\gamma_n).
\eeq
The above estimation implies that
\begin{align*}
H_{N}((r_n,\varphi_n))&=A_n-s_n^{\beta}(\sin \gamma_n-\sin v_n)\\
&=A_n+2s_n^{\beta} \sin\frac{\alpha_n}{2} \, \cos(\gamma_n+\frac{\alpha_n}{2})\\
&=A_n+\cO(s_n^{2\beta-1})=A_n+\cO(n^{-1}N^{-\frac{\beta}{\beta-1}})\\
&=C_N+\cO(n^{-1}N^{-\frac{\beta}{\beta-1}}).
\end{align*}

For $n\in[1,N_1]$, we denote $\cK_n$ as the curvature of the boundary at $x_n$. Then one can check that by (\ref{z1s}), the curvature satisfies
\beq\label{cKn}\cK_n=(\beta-1)s_n^{\beta-2}+\cO(s_n^{\beta-1})\eeq
Using (\ref{taun}), we get
\begin{align}\label{bartaun}
\frac{2\cK_n\tau_n}{\sin\gamma_n}&=\frac{2(\beta-1) (s_n^{2\beta-2}+s_n^{\beta-2}s_{n+1}^{\beta})}{\beta\sin\gamma_n\sin v_n}\nonumber\\
&=\frac{2(\beta-1) (s_n^{2\beta-2}+s_n^{2\beta-2}+s_n^{\beta-2}(s_{n+1}^{\beta}-s_n^{\beta}))}{\beta\sin\gamma_n\sin v_n}\nonumber\\
&= \frac{\bar a(\alpha_n^2+\alpha_{n+1}^2)}{\sin v_n\sin\gamma_n}+\cO(\alpha_n^3)\nonumber\\
&= \frac{\bar a(1+u_n^{-2})}{ w_n^2} +\cO(\alpha_n^3)\nonumber\\
&=\frac{2\bar a}{(2+\bar a)^2 n^2}+\cO(\alpha_n^3)
\end{align}
where we have used $\bar a=2(\beta-1)/\beta$, as well as (\ref{3.7}) and (\ref{snpower}).

Using (\ref{taun}), we know that
$$\frac{\tau_{n+1}}{\tau_{n}}=\frac{s_{n+1}^{\beta}+s_{n+2}^{\beta}}{s_n^{\beta}+s_{n+1}^{\beta}}\cdot\frac{\sin(\gamma_n+\alpha_n)}{\sin(\gamma_{n+1}+\alpha_{n+1})}.$$
Note that for $n\in[1,N_1]$,
$$E_1:=\frac{s_{n+1}^{\beta}+s_{n+2}^{\beta}}{s_n^{\beta}+s_{n+1}^{\beta}}=\frac{u_{n+1}^{-\frac{\beta}{\beta-1}}+1}{1+u_n^{\frac{\beta}{\beta-1}}}.$$
Thus we have
$$E_1=1-\frac{\beta}{(2\beta-1)n}+\cO(\alpha_n^2).$$
Moreover,
$$E_2:=\frac{\sin(\gamma_n+\alpha_n)}{\sin(\gamma_{n+1}+\alpha_{n+1})}=\frac{n\alpha_n}{(n+1)\alpha_{n+1}}+o(n^{-1})=1-\frac{2}{(2+\bar a)n}+\cO(\alpha_n^2). $$
Thus we have for $n\in[1,N_1]$,
\beq\label{ratiotau}\frac{\tau_{n+1}}{\tau_{n}}=E_1\cdot E_2=1-\frac{4}{(2+\bar a)n}+\cO(\alpha_n^2)=1-\frac{2\beta}{(2\beta-1)n}+\cO(\alpha_n^2).\eeq

Combining with items (1)-(5), we get for $1<n<N_1$,
$$\tau_n\sim \frac{1}{n^{\frac{2\beta}{2\beta-1}}N^{\frac{\beta}{(2\beta-1)(\beta-1)}}}.$$
For $N_1<n<N_3$, we have $\alpha_n\sim N^{-1}$ and $\gamma_n>\bar\gamma,$ which implies that
$$\tau_n\sim s_n\alpha_n\sim N^{-\frac{\beta}{\beta-1}}.$$
Thus the trajectory during each period in the corner series has
the  order of
$$\sum_{n=1}^{N_1}\tau_n\sim N^{-\frac{1}{\beta-1}},\,\,\,\,\,\,\sum_{n=N_1}^{N_2}\tau_n\sim N^{-\frac{1}{\beta-1}}.$$
 \end{proof}

 Due to the time reversibility of billiard dynamics, all the asymptotic formulas obtained for the entering period remain valid for the exiting period.

\section{Hyperbolicity of $(F,M)$}

\subsection{Stable/unstable cones}
We let  $M$ be the collision space defined as in (\ref{defM}), and we define the return time function $R: M\to\mathbb{N}$ as well as the induced map $F: M\to M$ as (\ref{Fdef}). Then $F$ preserves $\mu_M=\mu/\mu(M)$. Clearly, the induced map is indeed a dispersing billiard system. Thus using techniques in Chapter 4 of  \cite{CM}, one can obtain that the system $(F,M)$ is uniformly hyperbolic.

 In this section we investigate the expansion factors for vectors in unstable cones  of the induced system $(F,M,\mu_M)$. We denote $\cK(x)$ as the curvature of the boundary at the base point of $x$.

We first recall that the differential of the billiard map $\cF$ satisfies:
 \begin{align}\label{DTdiff}
D_x \cF=\frac{-1}{\cos\varphi_1}\left(\begin{array}{cc}\tau\cK(r)+\cos\varphi & \tau \\ \tau\cK(r)\cK(r_1)+\cK(r)\cos\varphi_1+\cK(r_1)\cos\varphi & \tau\cK(r_1)+\cos\varphi_1\end{array}\right).
\end{align}

Next we introduce the concept of the wave front. Let $V\in \cT_xM$ be a tangent vector. For $\eps>0$ small, let us consider  an infinitesimal curve $\gamma=\gamma(s)\subset M$, where $s\in (-\eps,\eps)$ is a parameter, such that $\gamma(0)=x$ and $\frac{d}{ds}\gamma(0)=V$. The forward (backward) trajectories of the points $y\in \gamma$, after leaving $M$, make a  bundle of directed lines in $Q$, which is called the forward (backward) ``wave front". Let $\cB=\cB(x)$ ($\cB^-=\cB^-(x)$) be the curvature of the orthogonal cross-section of the forward (backward)  wave front at the point $x$ with respect to the  vector $V$.	
Indeed we have
\beq\label{cBt}
\cB^-(x_1)=\frac{\cB(x)}{1+\tau(x)\cB(x)}=\frac{1}{\tau(x)+\frac{1}{\cB(x)}}
\,\,\,\text{ and }\,\,\,\,\,{\cB}(x)=\cB^-(x)+\tfrac{2\cK(x)}{\cos\varphi}.\eeq
where $x_1= \cF x$.

By our assumption on the table and the definition of $M$,  there exist $\tau_{\min}>0$,   $0<\cK_{\min}<\cK_{\max}<\infty$ such that for any $x\in M$,
\beq\label{takck}\tau(x)>\tau_{\min},\,\,\,\,\,\,\cK_{\min}\leq \cK(x)\leq \cK_{\max}.
\eeq
Denote by $\cV=d\varphi/dr$ the slope of the tangent line of $W$ at $x$. Then  $\cV$ satisfies
\beq\label{cVx}\cV= \cB^{-}\cos\varphi+\cK(r) = \cB
\cos\varphi-\cK(r).\eeq

 We use the cone method developed by Wojtkowski \cite{Wo85} for establishing hyperbolicity for phase points  $x=(r,\varphi)\in M$. In particular we study stable and unstable wave front. The relations in (\ref{cBt}) imply that a dispersing wave front remains bounded away from zero.
  \begin{defn}\label{defn:1}
  The unstable cone $\cC_{x}^u$ contains all tangent vectors based at ${x}$ whose images  generate dispersing wave fronts:
$$\cC^u_x=\{(dr, d\varphi)\in \cT_{x} M\,:\, \cK(r)\leq d\varphi/dr\leq	 \cK(r)+\tau_{\min}^{-1}
\}.
$$
Similarly the stable cones are defined as
 $$\cC^s_x=\{(dr, d\varphi)\in \cT_{x} M\,:\, -\cK(r)\geq d\varphi/dr\geq	 -\cK(r)-\tau_{\min}^{-1}
 \}.$$
 We  say that a smooth curve $W\subset M$ is an unstable
(stable) curve for the  system $(\cF,\cM)$ if at every point $x \in W$ the tangent line
$\cT_x W$ belongs in the unstable (stable) cone $C^u_x$
($C^s_x$). Furthermore, a curve $W\subset  M$ is an unstable
(resp. stable) manifold for the  system $(\cF,\cM)$ if $\cF^{-n}(W)$ is an unstable (resp.
stable) curve for all $n \geq 0$ (resp. $\leq 0$).
  \end{defn}

We consider a short unstable curve
$W\subset M_N$ with equation $\varphi=\varphi(r)$, for some $N\geq 1$. Let $x=(r,\varphi)\in W$ and $V=(dr,d\varphi)$ be
a tangent vector at $x$ of $W$.
Combining with (\ref{DTdiff}) and (\ref{MphiKo}),  we know for any $x\in M$, the slope of the tangent vector of $\cF W$ at $\cF x$ satisfies:
\beq\label{cV1}\cK(\cF x)+\frac{\cos\varphi_{_{K_0}}}{\tau_{\max}+\frac{1}{2K_{\min}}}\leq\cV(\cF x)=\cK(\cF x)+\frac{\cos\varphi(\cF x)}{\tau(x)+\frac{\cos\varphi}{\cK(x)+\cV(x)}}\leq \cK(\cF x)+\frac{1}{\tau_{\min}+\frac{\cos\varphi_{_{K_0}}}{2\cK_{\max}+\tau_{\min}^{-1}}}.\eeq
This implies that for any $x\in M$, the unstable cone $\cC^u_x$ is  already strictly invariant under $\cF$. Thus we get
$DF \cC^u_{F x}\subset \cC^u_{Fx}$.
 Similarly one can check that $D F(\cC^s_x)\supset \cC^s_{ F x}$.

Next we will show that  any unstable vector in $\cC^u$ gets uniformly expanded under $DF$. We introduce two metrics on the tangent space $\cT M$. Let $x\in M_N$, for any $N\geq 1$.
\begin{enumerate}
\item[(I)]
The first one is the so-called \emph{p-metric} on  vectors $dx = (dr, d\varphi)$ by
\beq\label{pnorm} |dx|_p = \cos \varphi \, |dr|.\eeq
 Put
$\cF x=(r',\varphi')$ and $dx'=(dr',d\varphi')=D \cF (dx) $, for $n\geq 1$.  The expansion factor is \beq
\frac{|D\cF (dx)|_p}{|dx|_p}= 1+\tau(x){\cal
	\cB}(x)\geq 1+\frac{\tau(x)\cK(r)}{\cos\varphi}.
			 \label{DTvu1}
\eeq

\item[(II)] Now consider the expansion factor in the \emph{Euclidean metric} $|dx|^2 =
(dr)^2 + (d\varphi)^2$.

 Note that for any $x=(r,\varphi)\in M_N$, with $N$ large enough,  $\cos\varphi$  is approximately $1$ by our assumptions on the billiard table. Moreover, for any $dx=(dr,d\varphi)\in\cC^u_x$,  we have  $d\varphi/d r$ is approximately $1$. Thus  for any vector $dx\in \cC^u_x$,
 \beq
\label{penorm}
\frac{|D\cF (dx)|}{|dx|} =
\frac{|D\cF (dx)|_p}{|dx|_p}\,
	\frac{\cos \varphi}{\cos\varphi'}\, \frac{\sqrt{1+(\frac{d\varphi'}{dr'})^2}}{\sqrt{1+{(\frac{d\varphi}{dr})^2}}}.\eeq
\end{enumerate}

 Hence, the $p$-norm and the Euclidian norm are equivalent for tangent vectors  at $x$ and $\cF x$.
We let $x\in M_N$ and define $x_n=(s_n,\varphi_n)=\cF^n x$, for $n=1,\cdots, N$, corresponding to iterations for a corner series of length $N$ as introduced in the above section.

\begin{proposition}\label{KBbd} For any $N>K_0$, any unstable curve $W\subset M_N$  such that $F^{-1}W$ is also unstable, let  $x=(r,\varphi)\in W$.   The total expansion factor for unstable
vectors in the course of the corner series of $N$ collisions has lower bound
$$\Lambda(x):=\frac{|D F (dx)|}{|dx|}\geq CN^{1+\frac{\beta}{(2\beta-1)(\beta-1)}},$$where $C>0$ is a constant. Its precise asymptotic is
$$\Lambda(x)\sim N^{1+\frac{\beta}{(2\beta-1)(\beta-1)}}\left(1+\frac{N^\frac{-\beta}{2\beta-1}}{\cos\varphi_1}\right)\left(1+\frac{N^{-\frac{\beta(\beta-2)}{(2\beta-1)(\beta-1)}} } {\cos\varphi_N}\right).$$
 \end{proposition}
The proof of this proposition is rather lengthy, so we put it in the appendix.

\section{Distribution of the return time function}
 In this section, we use the results of the previous sections to analyze the distribution of the return time function $R$, together with its level set $M_n$, $n\geq K_0$,
which consists of points whose trajectories go down a cusp and experience
there a corner series of exactly $n-1$ collisions.
We will use standard facts of the theory of dispersing billiards \cite{BSC90,
BSC91, C99, CM,CM05}. For example, the domains $M_n$ are bounded by singularity curves of the map $F$. These singular curves are made of unstable curves and the preimage of $\partial M$. Due to the time-reversibility of the billiard
dynamics, $F M_n$ is obtained by reflecting $M_n$ across the line
$\varphi=0$. Moreover, a point $(r, \varphi)$ is a singularity point for the map $F$ (i.e. $F(r,\varphi)$ or its differential  is not well-defined ) if
and only if $(r, -\varphi)$ is a singularity point for its inverse $F^{-1}$ (i.e. $F^{-1}(r,-\varphi)$ or its differential is not well-defined).

By our assumption, we know the singular trajectory running out of the cusp at $P$ will land on the point $D$ on the opposite side $\Gamma_3$ perpendicularly. Let $x_D=(r_D, 0)$, where $r_D$ is the $r$-coordinate of the singular point $x_D$. Indeed $x_D$ belongs to a singular curve, which we call  $s_0$, that is made of all  grazing collisions on $\partial Q$. One can check using (\ref{DTdiff}) that the slope of the tangent vector at $x=(r,\varphi)\in s_0$  satisfies
$$d\varphi/dr=-(K(x)+\cos\varphi/\tau(x)).$$
Thus in the vicinity of $x_D$, the curve  $s_0$ can be approximated by a line with slope $-K(x_D)-l_D^{-1}$, where $l_D$ is the distance between the cusp $P$ and the base point of $x_D$ in the billiard table.

The singular curves of $F$ near $x_D$ consist of   two
symmetric  sequences of singularity curves $\{s'_{n}\}$
and $\{s''_{n}\}$, approaching  $x_D$ from both sides of $s_0$.  More precisely, $s'_n$ consists of points whose trajectories enter the cusp by hitting $\Gamma_1$ first, and the last collision in the corner series is grazing. Similarly,  $s_n''$ hits  $\Gamma_2$ first, where the last collision  is grazing when exiting the cusp.  Denote by $M'_n$ the strip bounded between
$s'_{n}$, $s'_{n+1}$; and $M''_n$ bounded by $s''_n$ and $s''_{n+1}$. Then  $M_n=M'_n\cup M''_n$. By the symmetric property, it is enough to concentrate on $M'_n$.

\begin{figure}[h]
\center \psfrag{M}{$M$}
\psfrag{g1}{\scriptsize$pi/2$}
\psfrag{g3}{\scriptsize$\pi/2$}
\psfrag{xp}{\scriptsize$x_D
$}
\psfrag{Hn}{\scriptsize$H_{n}$}\psfrag{fMn''}{\scriptsize$\cF M_n''$}
\psfrag{mn''}{\scriptsize$M_n''$}\psfrag{cM}{\scriptsize$\cM$}\psfrag{M}{\scriptsize$$}
\psfrag{mn'}{\scriptsize$M_n'$}\psfrag{rf}{\scriptsize$r_f$}
\psfrag{s0}{\scriptsize$s_0$}\psfrag{fmn'}{\scriptsize$\cF M_n'$}
\psfrag{g'1}{\scriptsize$\varphi$}\psfrag{r}{\scriptsize$r$}\psfrag{s2}{\scriptsize$S_m$}
\includegraphics[width=6in,height=3in]{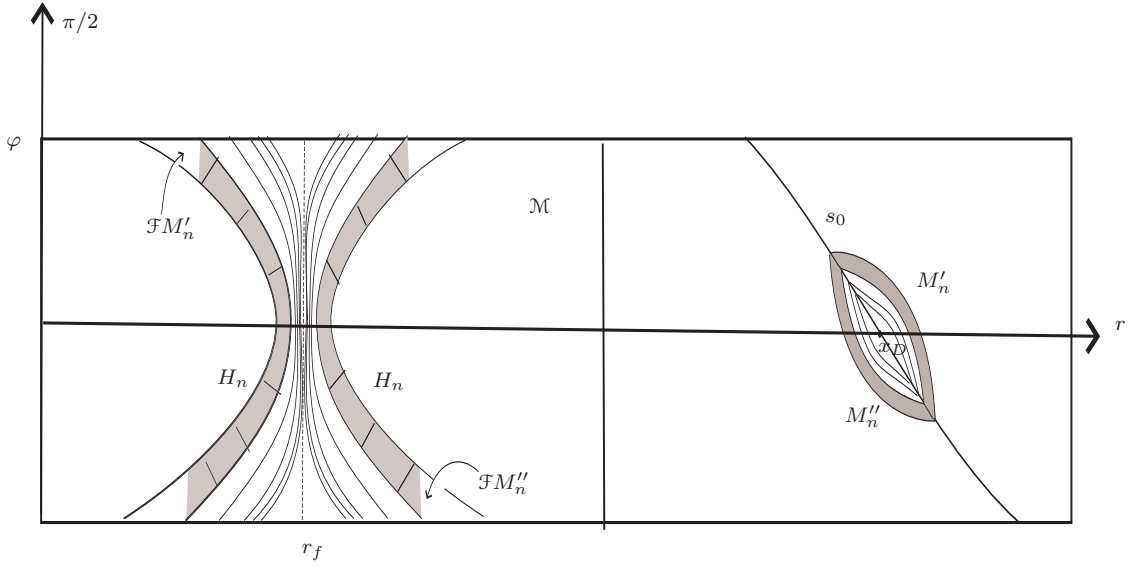}
\renewcommand{\figurename}{Figure.}
\caption{Singularity curves of $F$ in the  vicinity of $x_D$; and their forward images under $\cF$ near the cusp in $\cM$ (that are bounded by $H_n$).}\label{FigSing1}
\end{figure}

Fig \ref{FigSing1}. shows the structure of the singular curves near $x_D$.
In order to determine the rates of the
decay of correlations we need certain quantitative estimates on
the measure of the regions $M_n$, or $\{R\geq n\}$.
Note that for any $N\geq K_0$, the forward images $\{\cF^k M_N, k=1,\cdots, N-1\}$ fill entirely the regions squeezed between the two curves $H_N$ and $H_{N+1}$. By Proposition \ref{prop3}, we know that
the set $\cup_{n=1}^{N-1}  \cF^n M_N'\subset\cM$  is bounded by the  line $r=r_f$, $\varphi=\pi/2$  and a curve described implicitly by the equation of $H_N$:
 \beq\label{rbeta}r^{\beta}= \frac{C_N}{ \sin\varphi}(1+\cO(r^{2\beta-1}C_N^{-1}))\eeq

Equivalently $H_N$ has equation given by:
$$r= \frac{C_N^{\frac{1}{\beta}}}{ (\sin\varphi)^{\frac{1}{\beta}}}+\cO\left(\frac{r^{2\beta-1}}{C_N^{1-\frac{1}{\beta}}\sin^{\frac{1}{\beta}}\varphi}\right)$$
 We extend the definition of $R$ from $M$ to $\cM$, such that for any $x\in \cM$,
  $$R(x)=\min\{n\geq 1\,: \cF^n x\in M\}.$$ Using Propostion \ref{prop3}, we have for $N>K_0$,
 \begin{align}\label{muRn}
 \mu(x\in \cM\,:\, R\geq N)&=\sum_{m\geq N}\sum_{k=0}^{m-N} \mu(\cF^k M_m)\sim \sum_{m\geq N}\sum_{k=0}^{m-1} \mu(\cF^k M_m)\nonumber\\
 &\sim\int_{0}^{\pi}\frac{C_N^{\frac{1}{\beta}}}{ \sqrt[\beta]{\sin\varphi}}\,  \sin\varphi\,d\varphi=C_N^{\frac{1}{\beta}}\int_{0}^{\pi} (\sin\varphi)^{1-\frac{1}{\beta}}\,d\varphi\sim N^{-\frac{1}{\beta-1}}.\end{align}

Another method of calculation of the measure of $M_N$ relies on  the
factor of expansion of unstable manifolds $W \subset M_N$  under the map $F$. We use Proposition \ref {KBbd} to estimate $M_N$.

\begin{lemma}\label{muC} For any $N\geq 1$, $M_N$  has measure  $\sim N^{-2-\frac{1}{\beta-1}}$. Thus $\mu_M(x\in M\,:\, R\geq N)\sim N^{-1-\frac{1}{\beta-1}}$; and  $\mu(x\in \cM\,:\, R\geq N)\sim N^{-\frac{1}{\beta-1}}$.
\end{lemma}

\begin{proof}

To determine the dimensions of the strips $M_N'$, observe that the two intersection
 points in $s'_N$ with the curve $s_0$ are  located farthest from the central point
$x_D$, which are made by trajectories whose very first collision in the cusp is grazing. Let $x$ be such an end point, and $x_1=\cF x$. By our assumption, we know that the tangent vector at $x_1=(r_1,\varphi_1)$ makes an angle approximately $r_1^{\beta-1}$ with the horizontal line. By Proposition \ref{prop3}, we know that $\alpha_1\sim N^{-\frac{\beta}{2\beta-1}}$. Thus we conclude
that the trajectory originates at the distance $\sim N^{-\frac{\beta}{2\beta-1}}$ from the point $x_D$. Thus
the diameter of $M_N$ (i.e. the `length' of these strips) is $\sim N^{-\frac{\beta}{2\beta-1}}$.

Due to the time-reversibility of the billiard dynamics, the
singular curves in $F M$ have a similar structure. Furthermore,  the  short sides of $M_N'$ stretch completely  under $F M_N'$,  and are transformed into long sides of $FM_N'$.
Let $W\subset M_N'$ be a  short unstable curve  that stretches completely in $M_N'$ between two long sides. Then $|FW|\sim N^{-\frac{\beta}{2\beta-1}}$.
Thus by Proposition  \ref{KBbd}, the expansion factor on $W$ is  $\sim N^{1+\frac{\beta}{(2\beta-1)(\beta-1)}}$. This implies that the width is:
\beq\label{hmk}
|W|\sim N^{-\frac{\beta}{2\beta-1}}/N^{1+\frac{\beta}{(2\beta-1)(\beta-1)}}\sim N^{-1-\frac{\beta}{2\beta-1}-\frac{\beta}{(2\beta-1)(\beta-1)}}.\eeq
 Note that the length of $M_N'$ is $\sim N^{-\frac{\beta}{2\beta-1}}$, and  the density on  $M_N'$
is $\sim \cos\varphi\sim 1$, as the  collision at $x$ is almost perpendicular.
 According to  (\ref{hmk}),  the
measure of $M_N$ is of order
$$\mu_M(M_N)\sim N^{-1-\frac{\beta}{(2\beta-1)(\beta-1)}}\cdot N^{-\frac{2\beta}{2\beta-1}}=N^{-2-\frac{1}{\beta-1}}.$$

\end{proof}

 It is
immediate that for any sufficiently  large $N$,  \begin{align}\label{Mpol2}
 \mu_M&(x\in M:\ R(x)\geq N)
	\sim    \,  N^{-1-\frac{1}{\beta-1}}\,\,\,\,\,\,\text{ and }\,\,\,\,\,\,\mu(x\in \cM:\ R(x)\geq N)
	\sim    \,  N^{-\frac{1}{\beta-1}}.\end{align}

This verifies condition (\textbf{F2}) with $a=\frac{1}{\beta-1}$.

\section{Exponential decay rates for the reduced system}

According to the general scheme proposed in Section 2, we  need to check condition (\textbf{F1}), i.e.,  prove that the induced system $(F,M,\hat\mu)$ enjoys exponential decay of correlations.
Here we use a simplified  method to prove exponential decay of
correlations for our reduced billiard map. It is mainly based on
recent results in \cite{ Y98, C99, CZ}.

Since billiards have singularities, if the orbit of $x$ approaches  the singularity set $\cS_1$ too fast under $F$, then $x$ may not have a stable/unstable manifold.  The situation is kind of complicated here as we have  accumulated sequences of new types of singular curves. Indeed we will first show that a small neighborhood of the singular set has small measure.
\begin{lemma}  For any $\delta>0$, the $\delta-$neighborhood of $\cS_{\pm 1}$ has measure:
 \beq\label{mu0m}\mu(B_{\delta}(\cS_{\pm 1}^{ }))\leq C  \delta^{\frac{\beta(2\beta-1)}{3\beta^2-3\beta+1}}.\eeq
  Here $B_{\delta}(\cS_{\pm 1}^{ })=\{x\in M\,:\, d_M(x,\cS_{\pm 1}^{ })\leq \delta\}$ for any $\delta>0$, and $C>0$ is a constant, where $d_M(\cdot,\cdot)$ is the distance in $M$.
\end{lemma}
\begin{proof}
For any given small $\delta>0$,
we first need to find the smallest $N_{\delta}$ such that $\cup_{n\geq N_{\delta}} M_n\subset \mu_M(B_{\delta}(\cS_{ 1}^{ }))$.
According to Lemma \ref{muC}, the width of $M_N$ is approximately $\cO(N^{{-1-\frac{\beta}{2\beta-1}-\frac{\beta}{(2\beta-1)(\beta-1)}}})$. Thus we get $N_{\delta}=\delta^{-\frac{1}{1+\frac{\beta}{2\beta-1}+\frac{\beta}{(2\beta-1)(\beta-1)}}}$. This implies that
$$\mu_M(\cup_{n\geq N_{\delta} }M_n)\sim \delta^{\frac{\beta(2\beta-1)}{3\beta^2-3\beta+1}}.$$
Thus we have
$$\mu(B_{\delta}(\cS_{\pm 1}^{ }))\leq \mu(\cup_{n\geq N_{\delta} }M_n)+ C\sum_{n=1}^{N_{\delta}} \delta n^{-\frac{\beta}{2\beta-1}}\sim \delta^{\frac{\beta(2\beta-1)}{3\beta^2-3\beta+1}}.$$

\end{proof}

The above lemma  implies that almost every point in $M$ has a regular stable (resp. unstable) manifold and there are plenty of reasonable long stable (resp. unstable)  manifolds for $F$.   These manifolds are denoted as $\cW^{s/u}$. If the forward image of $x\in W\in\cW^u$ is almost tangential, then the expansion factors along  $W$ may be highly nonuniform. To overcome this difficulty we divide $\cM$ into horizontal strips as introduced in \cite{BSC90, BSC91}.  More precisely, one divides $\cM $ into countably many sections (called
\emph{homogeneity strips}) defined by
$$
	\bH_k=\{(r,\varphi)\in \cM \colon \pi/2-k^{-2}<\varphi <\pi/2-(k+1)^{-2}\},
$$
and
$$
	\bH_{-k}=\{(r,\varphi)\in \cM \colon -\pi/2+(k+1)^{-2}<\varphi < -\pi/2+k^{-2}\},
$$
for all $k\geq k_0$ and \beq \label{bbH0}
	\bH_0=\{(r,\varphi)\in \cM \colon -\pi/2+k_0^{-2}<\varphi <
	\pi/2-k_0^{-2}\},
\eeq where $k_0 \geq 1$ is a fixed (and usually large) constant,  whose value will be chosen to guarantee the one-step expansion  -- for details, see the end of the proof of Lemma  5.56 in \cite{CM05} for such a  choice of $k_0$. We now add the boundary of these homogeneous strips in the the singular set $\cS^H_{\pm 1}:=\cS_{\pm 1}\cup \{\partial \bH_k\cup F^{\pm 1}\bH_k, k\geq k_0\}$. We denote the resulting collection of stable/unstable manifolds as $\cW^{s/u}_H$, and call them the homogeneous invariant manifolds.

\begin{lemma}\label{regularity} For the induced map $(F,M)$, the invariant manifolds in $\cW^{s/u}_H$ have the regularity properties: bounded curvature, distortion bounds and absolute continuity.
\end{lemma}
The proof of this lemma follows the arguments  in \cite{CM}~Chapter 5 as well as \cite{C99}, as  the induced map $(F,M)$ is essentially the dispersing billiards (with corner points).  Thus we will not repeat here.
\begin{lemma}\label{onestep}(\textbf{One-step expansion estimate}) Assume $\beta\in(2, \infty)$. Let $W$ be a short unstable curve in $M$ and $\{W_i\}$ be the collection of smooth components in $W$. Then \beq
   \liminf_{\delta_0\to 0}\
  \sup_{W\colon |W|<\delta_0}\sum_{i\geq 1}
\frac{1}{\Lambda_i}<1,
	  \label{step1}
\eeq where the supremum is taken over unstable curves $W\subset M$ and $\Lambda_i=\frac{|FW_i|}{|W_i|}$,
$i \geq 1$, denote the minimal local expansion factors of the
connected component $W_i$ under the map $F$.
\end{lemma}

\begin{proof}
Let $W\subset M$ be an unstable curve.  Note that the upper bound of (\ref{step1}) is only achieved when $W$ intersects one of the accumulating sequences of the singular set.

    We consider the worst case by assuming $W$ touches the singular point $x_D$, and intersects $M_N$, for $N\geq n_0$, where $n_0$ depends on the length of $W$. Note that  $W$ crosses  $s_{N+1}$, the boundary of cell $M_N$, for some $N\geq n_0$. Since $s_{N+1}$ consists of points whose last iteration in the corner series is tangential to the boundary of the table, then it must belong to a cell $ \cF^{-N}\bH_k$, for some $k\geq k_N$. Moreover, if $x\in s_{N+1}$ is very close to the curve $s_0$, then its first collision must  cross some region $ \cF^{-1}\bH_m$, for some $m\geq m_N$.
              Define $W_{N,m,k}=W\cap \cF^{-N}\bH_k\cap \cF^{-1}\bH_m$.
               Let $x\in W_{N,m,k}$, and $x_n=\cF^n x$, for $n\geq 1$, then  Proposition \ref{KBbd} implies that  the expansion factor satisfies:
          \begin{align*}
          \Lambda_{N,m,k}:=\frac{|FW_{N,m,k}|}{|W_{N,m,k}|}&\sim N^{1+\frac{\beta}{(2\beta-1)(\beta-1)}}\left(1+\frac{N^\frac{-\beta}{2\beta-1}}{\cos\varphi_1}\right) \cdot \left(1+\frac{N^{-\frac{\beta(\beta-2)}{(2\beta-1)(\beta-1)}} } {\cos\varphi_N}\right)\\
          &\sim N^{1+\frac{\beta}{(2\beta-1)(\beta-1)}}\left(1+m^2N^\frac{-\beta}{2\beta-1}\right) \cdot \left(1+k^2N^{-\frac{\beta(\beta-2)}{(2\beta-1)(\beta-1)}} \right).
          \end{align*}

          Next we will find $m_N$, which is the smallest integer such that $W\cap \cF^{-1}\bH_m$ is not empty. Note that the expansion factor for $\cF W$ is approximately $1$, thus $\cF M_N'$ is a cell  bounded by $\varphi=\pi$, with $r$-dimension $\sim N^{-1}$ and $\varphi$-dimension $\sim N^{-1}$. Thus it intersects infinitely many homogeneous strips $\bH_m$, with $m\geq m_N$. Thus $k_N\sim N$, which implies that $m_N\sim N$. Further images $\cF^i M_N'$ moves away from $\varphi=\pi$, they only intersects homogeneous strips $\bH_m$, with $m<m_N$, for $i=2,\cdots N_1$.
          By the symmetric property of the billiard table, when $i$ approaches  $N$, similar patterns repeat, with $\cF^N M_N'$ intersecting infinitely many $\bH_k$, for $k\geq m_N$.
Thus we have
\begin{align*}
\sum_{N\geq n_0}\sum_{m\geq m_N}\sum_{k\geq k_N}\frac{1}{\Lambda_{N,m,k}}&\leq C \sum_{N\geq n_0}\sum_{m\geq m_N}\sum_{k\geq k_N} \frac{N^{-1-\frac{\beta}{(2\beta-1)(\beta-1)}}}{\left(1+m^2N^\frac{-\beta}{2\beta-1}\right) \cdot \left(1+k^2N^{-\frac{\beta(\beta-2)}{(2\beta-1)(\beta-1)}} \right)}\\
 &\leq C n_0^{-\frac{2\beta}{2\beta-1}}.
\end{align*}

Note that we have assumed $W$ intersects $M_N$, for all $N\geq n_0$, thus by (\ref{hmk}),
$$|W|\sim \sum_{N\geq n_0} N^{-1-\frac{\beta^2}{(2\beta-1)(\beta-1)}} \sim n_0^{-\frac{\beta^2}{(2\beta-1)(\beta-1)}}.$$

 Combining the above estimations, we have
 $$\sum_{N\geq n_0}\sum_{m\geq N}\sum_{k\geq N}\Lambda^{-1}_{N,m,k}\leq C |W|^{\frac{2(\beta-1)}{\beta}}.$$
   Thus by taking $|W|$ small, we can make the above sum $<1$.\end{proof}
Given  an unstable curve $W$, a point $x\in W$ and an integer $n\geq 0$, we denote by $r_{n}(x)$ the distance between $F^n x$ and the
boundary of the homogeneous component of $F^n W$ containing $F^n x$. Clearly $r_n(x)$ is a function on $W$ that characterizes the size of the smooth components of $F^n W$.
We first state  the Growth Lemma, proved in \cite{CZ09}, which is key in the analysis hyperbolic systems with singularities. It expresses the fact that the expansion of unstable curves dominates the cutting by singular curves, in a uniform fashion for all sequences. The reason behind this fact is that unstable curves expand at a uniform exponential rate, whereas the cuts accumulate at only a finite number of singular points. The following Growth Lemma can be derived directly from Lemma \ref{onestep} -- the One-step Expansion Estimates, see \cite{CM}, \cite{CZ09} for details.
\begin{lemma}\label{lmmstep} (Growth Lemma). There exist uniform constants $C_{\bg},c > 0 $ and $\vartheta\in (0,1)$, $q=\frac{\beta(2\beta-1)}{3\beta^2-3\beta+1}$,  such that,
for any probability measure $\nu$ supported on an unstable curves $W$ with positive density $d\nu/dm_W\in \cH_{\gamma}$ for some $\gamma\in (0,1)$, and $n\geq 1$ :
\beq\label{firstgrowth1}
F^n\nu(r_n<\eps)\leq C_{\bg}\,\vartheta^{n}
\nu(r_n<\eps)^q+c\eps,
\eeq
where $m_W$ is the Lebesgure measure on $W$.
\end{lemma}

In \cite{C99,CZ,CZ09}, the following lemma was proved.
\begin{lemma}\label{lem:1}
If the induced billiard map $F$ satisfies (\ref{step1}), and the unstable manifolds have regularities: bounded curvature, distortion bounds and absolute continuity, then there is a hyperbolic horseshoe $\Delta_0\subset M$ such
that \beq\label{Yexp1}
	\mu_M\bigl(x\in M:\ R(x;F,\Delta_0)>m\bigr)\leq
	\,C\theta^m\quad\quad \forall m\in \mathbb N,
\eeq for some $\theta<1$, where $R(x; F, \Delta_0)$ is the return
time of $x$ to $\Delta_0$ under the map $F$. Thus
 the map $F:M\to M$ enjoys exponential decay of correlations. More precisely,
for every pair of dynamically H{\"o}lder continuous functions
 $f,g\in \cH_{\gamma}$ and $n \geq
 1$,
\beq  \label{expbound} 	
\cC_n(f,g,F,\mu_M)\leq C \|f\|_{C^{\gamma}}\|g\|_{C^{\gamma}} \theta^{n},
\eeq where  $C>0$ is a uniform constant.
\end{lemma}

The above results can be extended to  variables made at
multiple times. Let $f_0, f_1, \ldots, f_k\in
\cH_{\gamma}$, and
 $\|f_i\|_{\infty}=\|f\|_{\infty}, $ $i=1, \ldots, k$. Consider
 the product
$\tf=f_0\cdot (f_1\circ F)\cdots (f_k\circ F^k).$
Furthermore, let $g_0, g_1, ..., g_k\in \cH_{\gamma}$,
and $\|g_i\|_{\infty}=\|g\|_{\infty}$, $i=1, ..., k$. Consider
the product $\tg=g_0\cdot (g_1\circ F)\cdots (g_k\circ F^k).$
Then we can estimate the correlations between observables $\tf$
and $\tg$.
\begin{theorem} There exists $C>0$, such that for all $n\geq
0$,
\begin{equation*}
\cC_n(\tf,\tg,F,\mu_M)\leq C \|\tf\|_{C^{\gamma}}\|\tg\|_{C^{\gamma}}
\theta^n,
\end{equation*}
where $\theta$ is the same as in (\ref{expbound}).
\end{theorem}

Hence we conclude that for $\beta\in(2, \infty)$, the return
map $F\colon M\to M$ has exponential mixing rates by the above Theorem, thus both condition (\textbf{F1}) and (\textbf{F2}) are verified.
The following lemma was proved in \cite{CZ}.
\begin{lemma} \label{LmMain2}
For systems under assumptions ({\textbf{F1-F2}}), and for the billiard map $\cF:\cM\to\cM$ and any piecewise
H\"older continuous functions $f,g\in \cH(\gamma)$ on $\cM$,  the correlations
(\ref{Cn})  decay as
\beq \label{main2}
   |\cC_n(f,g,\cF,\mu)|\leq\,C\|f\|_{C^{\gamma}}\|g\|_{C^{\gamma}}n^{-a}(\ln n)^{1+a},\eeq
   for some constant $C>0$.
\end{lemma}

 Now we see that  Theorem~\ref{TmMain}  should
follow from (\textbf{F1}) - (\textbf{F2}) for $a=\frac{1}{\beta-1}$, except for the extra logarithmic factor.
To improve the upper bound for the decay rates, one needs to analyze the statistical properties of the return time function.  In \cite{CZ3}, the upper bound for decay rates of correlations was improved by dropping the logarithmic factor.

\section{Proof of the main Theorems}

   A general strategy for estimating the correlation function $\cC_m(f, g,\cF,\mu)$ for
systems with weak hyperbolicity was developed in \cite{CZ, CZ3}.

 Note that Proposition \ref{prop3} also leads to the the following fact about transitions between cells with different indices.
\begin{proposition}\label{Cor4} There exist positive constants $c_1< c_2$, such that for any $n\geq 1$, if
  $ M_{m}\cap F	 M_{n}\neq\emptyset$  then
\beq\label{tran}c_1\sqrt[\beta] {n^{\beta-1}}\leq m\leq c_2
\sqrt[\beta-1] {n^{\beta}}.\eeq
Moreover, for any $m\in [c_1n^{\frac{\beta-1}{\beta}}, c_2n^{\frac{\beta}{\beta-1}}]$, the transition probability satisfies
$$\mu(M_m| F(M_n))\sim m^{-1-\frac{\beta^2}{(\beta-1)(2\beta-1)}} n^{\frac{\beta}{2\beta-1}}.$$
\end{proposition}
\begin{proof}  Without loss of generality  we only consider the singular curves near $x_D$. In particular note that the curves in $\cS_1$ and $\cS_{-1}$ are symmetric about $r_D$ in the vicinity of $x_D$.
Let $W\subset M_n$ be a  short unstable curve  that stretches completely in $M_n$ between two long sides. Then $|FW|\sim n^{-\frac{\beta}{2\beta-1}}$. On the other hand, note that by (\ref{hmk}), the width of cell $M_n$ is of order $\sim n^{-1-\frac{\beta^2}{(2\beta-1)(\beta-1)}}$. We assume $n_1, n_2$ are the two extreme indices, such that $$n_1=\min\{ m\geq 1\,:\, FM_n\cap M_m \neq \emptyset\},\,\,\,\,\,\,\,\,n_2=\max\{ m\geq 1\,:\, FM_n\cap M_m \neq \emptyset\}.$$
Then
\begin{align*}
|FW|& = \sum_{m\geq n_1} |FW\cap M_m| \\
&\sim \sum_{m\geq n_1} m^{-1-\frac{\beta^2}{(2\beta-1)(\beta-1)}} =n_1^{-\frac{\beta^2}{(2\beta-1)(\beta-1)}}.
\end{align*}
Now using the fact that $|FW|\sim n^{-\frac{\beta}{2\beta-1}}$, we can solve for $n_1 \sim n^{\frac{\beta-1}{\beta}}$.

By  time - reversibility, one can verify that  for the largest index $n_2$, if $F x\in M_{n_2}$, then $n$ is the minimal index, such that  $ F^{-1} M_{n_2} \cap M_n$ is not empty. Thus  above estimation implies $n\sim n_2^{\frac{\beta-1}{\beta}}$, which is equivalent to $n_2\sim n^{\frac{\beta}{\beta-1}}$.

Next we calculate the transition probability from $M_n$ to $M_m$, for $m\in [n_1, n]$. Note that $FM_n\cap M_m$ can be approximated by a rectangle  with dimensions given by the width of $FM_n$ and $M_m$. More precisely, $FM_n\cap M_m$ can be approximated as a rectangle with ``width'' (its
$r$-dimension) $\sim n^{-1-\frac{\beta^2}{(2\beta-1)(\beta-1)}}$, ``height'' (the
$\varphi$-dimension) $\sim m^{-1-\frac{\beta^2}{(2\beta-1)(\beta-1)}}$ and density weight $\cO(1)$.  This implies that
$$\mu(M_m|FM_n)\sim \frac{  n^{-1-\frac{\beta^2}{(2\beta-1)(\beta-1)}} m^{-1-\frac{\beta^2}{(2\beta-1)(\beta-1)}}}{\mu(M_n)} \sim m^{-1-\frac{\beta^2}{(2\beta-1)(\beta-1)}} n^{\frac{\beta}{2\beta-1}}.$$
\end{proof}

 Although it follows from the above lemma that some points in $M_n$ are mapped to cells with	higher indices, one can show that	 most of the points in $M_n$ indeed have images which belong to cells with much smaller indices.
\begin{lemma}\label{lem3} For any small $e>0$, we define $D_n(e)=\bigcup_{m\geq n^{\frac{\beta-1}{\beta}+e}} M_m$.
Then for any $n$ sufficiently large,
$$\mu(D_n(e)|FM_n)\sim n^{-\frac{e\beta^2}{(2\beta-1)(\beta-1)}} \mu(M_n).$$
In addition, let $\beta_0=\frac{3+\sqrt{5}}{2}$, then $\mathbb{E}(R\circ |M_n)(x)=e_n\bI_{M_n}(x)$, for some  $e_n>0$, with $e_n\leq C n^{1-\gamma}$, where $C>0$ is a constant and $\gamma= \frac{1}{(\beta-1)^2}$ for $\beta> \beta_0$ and $\gamma= \frac{1}{\beta}$ for $\beta\in (2, \beta_0]$, And $\gamma=\frac{1}{\beta}+\eps_0$ for $\beta=\beta_0$, where $\eps_0\in (0,(\frac{1}{\beta-1}-\frac{1}{\beta})/100)$.\end{lemma}
\begin{proof}
For any $n$ large, we denote $\beta_n:=\{m\in [c_1 n^{\frac{\beta-1}{\beta}}, c_2n^{\frac{\beta}{\beta-1}}]\}$ as the index set of $m$, such that $M_m\cap F M_n$ is not empty.
 Combining with	 Proposition \ref{Cor4}, we know for any $m\in \beta_n$, $$\mu(M_m|F M_n)\sim m^{-1-\frac{\beta^2}{(2\beta-1)(\beta-1)}} n^{\frac{\beta}{2\beta-1}}.$$

Thus the conditional expectation of $R\circ F$ on $M_n$ satisfies
\begin{align*}
\mathbb{E}(R\circ F| M_n)&\sim\sum_{m\in\beta_n} m \cdot \mu(M_m|F M_n)\cdot \bI_{M_n}
\sim \sum_{m\in \beta_n} m \cdot m^{-1-\frac{\beta^2}{(2\beta-1)(\beta-1)}} \cdot n^{\frac{\beta}{2\beta-1}}\cdot \bI_{M_n}.\end{align*}
We let $\beta_0=\frac{3+\sqrt{5}}{2}$. Note that for $\beta=\beta_0$, the $r$-dimension of the cell $M_n$ is $\sim n^{-1-\frac{\beta^2}{(2\beta-1)(\beta-1)}}= n^{-2}$. Thus for $\beta\in (2,\beta_0)$,
 $$\mathbb{E}(R\circ F| M_n)\sim R^{1-\frac{1}{\beta}}\bI_{M_n};$$ while for $\beta=\beta_0$, $$\mathbb{E}(R\circ F| M_n)\sim R^{1-\frac{1}{\beta}}\ln R\cdot \bI_{M_n}=R^{1-\frac{1}{(\beta-1)^2}}\ln R\cdot \bI_{M_n};$$ and for $\beta>\beta_0$, $$\mathbb{E}(R\circ F| M_n)\sim R^{1-\frac{1}{(\beta-1)^2}}\cdot \bI_{M_n}.
 $$ We take $\gamma=\frac{1}{(\beta-1)^2}$  for $\beta> \beta_0$; $\gamma=\frac{1}{\beta}+\eps_0$ for $\beta=\beta_0$; and $\gamma=\frac{1}{\beta}$ for $\beta<\beta_0$. Here $\eps_0\in(0,(\frac{1}{\beta-1}-\frac{1}{\beta})/100)$.
  Thus we have shown that $\mathbb{E}(R\circ F| M_n)=e_n\bI_{M_n}$, with $e_n=\cO(n^{1-\gamma})$, and $e_n>0$.

For any small $e>0$, we define $D_n(e)=\cup_{m\geq n^{\frac{\beta-1}{\beta}+e}} M_m$. Then we have
 $$\mu(D_n(e)| FM_n)=\sum_{k\geq n^{\frac{\beta-1}{\beta}+e}} \mu( M_k|F M_n) \sim \sum_{k\geq n^{\frac{\beta-1}{\beta}+e}} k^{-1-\frac{\beta^2}{(2\beta-1)(\beta-1)}} \cdot  n^{\frac{\beta}{2\beta-1}}=n^{-\frac{e\beta^2}{(2\beta-1)(\beta-1)}}.$$

\end{proof}
This lemma also implies, essentially, that a typical trajectory stays away from	those long corner series  most of the time. But as $\beta$ goes to $\infty$, more points in $F M_n$	 tend to enter long corner series more frequently.

\begin{lemma}\label{barMm}
For sufficently large $m$, any $b\geq 1$, there exists
$E_m\subset M_m$, with $\mu(M_m\setminus E_m)\leq
m^{-\frac{\beta^2}{2(\beta-1)(3\beta^2-3\beta+1)}}\mu(M_m)$, and any $x\in E_m$, $Fx, F^2x, ..., F^{b\ln
m} x$ all belong to cells with index less than $m^{1-\frac{1}{2\beta}}$.
\end{lemma}
\begin{proof} For  any $e>0$ small, it follows from  Lemma \ref{lem3},
$$
\mu\left(D_m(e)|F M_m\right)=\sum_{n=m^{\frac{\beta-1}{\beta}+e}}^{\infty}
\mu\left(M_n|F M_m\right)=\cO(m^{-\frac{e\beta^2}{(2\beta-1)(\beta-1)}}).$$
Below we choose $e=\frac{1}{2\beta}$, and denote $\cD_m:=D_m(\frac{1}{2\beta})$. Thus we can neglect points  $x\in M_m$ such that $F(x)\in M_m$ with
$n>m^{1-\eps}=m^{1-\frac{1}{2\beta}}$, for $\eps=\frac{1}{\beta}-e=\frac{1}{2\beta}$.
It remains to estimate the probability that points $y \in M_m$
 will come
up to $M_i$, for $i\geq m^{1-\frac{1}{2\beta}}$, within $\cO(\ln m)$ iterations of
$F$.

 Note that each cell $M_m$ has dimension $\sim m^{-\frac{\beta}{2\beta-1}}$ in the stable direction, dimension $\sim m^{-1-\frac{\beta^2}{(2\beta-1)(\beta-1)}}$ in the unstable direction, and measure $\mu(M_m)\sim m^{-2-\frac{1}{\beta-1}}$.  We first foliate $M_m$ with unstable curves $W_{{\alpha}}\subset M_m$ (where $\alpha$ runs through an index set $\cA$). These curves have length $|W_{\alpha}|\sim m^{-1-\frac{\beta^2}{(2\beta-1)(\beta-1)}}$. Let $\nu_m:=\frac{1}{\mu(M_m)}\mu|_{M_m}$ be the conditional measure of $\mu$ restricted on $M_m$. Let $\cW_m=\cup_{\alpha\in \cA} W_{\alpha}$ be the collection of all unstable curves, which foliate the cell $M_m$. Then we can disintegrate the measure $\nu_m$ along the leaves $W_{\alpha}$. More precisely,
for any measurable set $A\subset M_m$,
$$\nu_m(A)=\int_{\cA}\nu_{\alpha}(W_{\alpha}\cap A) \, d\lambda(\alpha),$$
where  $\lambda$ is the probability factor measure on $\cA$.  For each unstable curve $W_{\alpha}\in\cW$, if $F^m W_{\alpha}$  crosses $\cD_m$, then $F^l W_{\alpha}$ is cut into pieces by the boundary of cells in $\cD_m$. Moreover, the largest length of these pieces is $\sim m^{-(1+\frac{\beta^2}{(2\beta-1)(\beta-1)})(1-\frac{1}{2\beta})}= m^{-\frac{3\beta^2-3\beta+1}{2\beta(\beta-1)}}$.  According to the growth lemma \ref{lmmstep}, there exists $\vartheta\in(0,1)$, such that we have
\beq\label{growthest}F^l\nu_m(\cD_m)\leq C_{\bg}\,\vartheta^{l}
F^l\nu_m(\cD_m)^{\frac{\beta(2\beta-1)}{3\beta^2-3\beta+1}}+c m^{-\frac{3\beta^2-3\beta+1}{2\beta(\beta-1)}}.\eeq

Moreover Lemma \ref{lem3} implies that
$$F\nu_m(\cD_m):=\mu(R(F(x))>m^{1-\frac{1}{2\beta}} |R(x)=m)\leq C m^{-\frac{\beta}{2(2\beta-1)(\beta-1)}},$$
for some uniform constant $C>0$.

Now we apply (\ref{growthest}) to get for any $l=1,\cdots, b\ln m$,
\begin{align*}\nu_m(R(F^l(x))&>m^{1-\frac{1}{2\beta}} )=F^l\nu_m(\cD_m)\\
&\leq C_{\bg}\vartheta^l  F\nu_m(\cD_m)^{\frac{\beta(2\beta-1)}{3\beta^2-3\beta+1}}+cm^{-\frac{3\beta^2-3\beta+1}{2\beta(\beta-1)}}\\
&\leq C\vartheta^l m^{-\frac{\beta^2}{2(\beta-1)(3\beta^2-3\beta+1)}}+c m^{-\frac{3\beta^2-3\beta+1}{2\beta(\beta-1)}}.\end{align*}

Thus we have
$$\sum_{l=1}^{b\ln m}\nu_m(R(F^l(x))>m^{1-\frac{1}{2\beta}} )\leq C_1 m^{-\frac{\beta^2}{2(\beta-1)(3\beta^2-3\beta+1)}}.$$
This also implies that  for any large $m$, there exists
$E_m\subset (R=m)$, with $$\mu((R=m)\setminus E_m)\leq
m^{-\frac{\beta^2}{2(\beta-1)(3\beta^2-3\beta+1)}}\mu(R=m)$$ and any $x\in E_m$, $Fx, F^2x, ..., F^{b\ln
m} x$ all belong to cells with index less than $m^{1-\frac{1}{2\beta}}$.
\end{proof}

Now we are ready to prove Theorem 1.

		The tower in
$M$ can be easily and naturally extended to $\cM$, thus we get a
the Young's tower with the same base $\Delta_0 \subset M$; and
a.e.\ point $x\in \cM$ again properly returns to $\Delta_0$ under
$\cF$ infinitely many times. Consider the return times to $M$
under $\cF$ for $x\in \cM$. According to Lemma \ref{muC},
   \beq
	\mu(x\in \cM:\ R(x)>n)\sim
	\, \frac{1}{n^{\frac{1}{\beta-1}}},\quad\quad \forall n\geq 1.
	   \label{Mpol1}
\eeq

 For every $m\geq 1$ and $x\in\cM$ denote
$$
  r(x;m,M)=\#\{1\leq i\leq m:\ \cF^i(x)\in M\}.
$$
Let\begin{align*}
	A_m &=\{x\in \cM\colon\ R(x;\cF,\Delta_0)>m\},\\
	B_{m,b} &=\{x\in \cM\colon\ r(x;m,M) > b\ln m\},
\end{align*}
where $b>0$ is a constant to be chosen shortly.

By (\ref{Yexp1}), we know that
	$$
	\mu(A_m \cap B_{m,b})\leq\,C\cdot m\,\theta^{b\ln m}.
$$
Choosing $b=-\frac{2}{\ln\theta}$, then \beq\label{bt}{\rm const}\cdot
m\,\theta^{b\ln m}\leq {\rm const}\cdot m\,
\theta^{-\frac{2}{\ln\theta}\ln m}={\rm const}\cdot m^{-1}.\eeq

The set $A_m \setminus B_{m,b}$ consists of points $x\in \cM$
whose images under $m$ iterations of the map $\cF$ return to $M$
at most $b \ln m$ times but never return to the `base' $\Delta_0$
of Young's tower. Our goal is to show that $\mu(A_m \setminus
B_{m,b}) = \cO(m^{\frac{1}{1-\beta}})$.

Let $I=[n_0, n_1]$ be the longest interval, within $[1, m]$,
between successive returns to $M$. Without loss of generality, we
assume that $m-n_1 \geq n_0$, i.e.\ the leftover interval to the
right of $I$ is at least as long as the one to the left of it
(because the time reversibility of the billiard dynamics allows us
to turn  time backwards).
	  Due to Lemma \ref{barMm}, for a large portion of typical points $y \in M_{|I|}$ we have $F^t (y) \in M_{m_t}$,
where $m_t$ decreases exponentially fast. So there exists $c>0$,
such that $$ m/2\leq |I|+b|I|^{1-\frac{1}{2\beta}}\ln |I|\leq c|I|,$$ which gives
$|I|\geq \tfrac{m}{2c}$.

Let $G_m=\{x\in A_m\setminus B_{m,b}\,:\, |I|\geq
\tfrac{m}{2c}\}$. Thus it is enough to estimate the size of $G_m$.
Since for any $x\in G_m$, one of its forward images belongs to
$m_{|I|}$ with $ |I|\geq \tfrac{m}{2c}$. Applying the bound
Lemma \ref{muC} to the interval $I$ gives
\beq\label{ABslack}
	 \mu(G_m)
	 \leq C\,m\cdot\, m\,\cdot\, m^{-2-\frac{1}{\beta-1}}=\,C m^{\frac{1}{1-\beta}},
\eeq
 (the extra factors of $m$ must be included because the
interval $I$ may appear anywhere within the longer interval
$[1,m]$, and the measure $\mu$ is invariant).

 In terms of Young's tower $\Delta$, we
obtain \beq
	\mu(x\in\Delta:\ R(x;\cF,\Delta_0)>m)\leq
	\,C  m^{\frac{1}{1-\beta}},\quad\quad \forall m\geq 1.
	   \label{Ypol3}
\eeq  This completes the proof of the theorem \ref{TmMain}.

\medskip
\medskip

\medskip
\medskip

\section{General models with cusps}
In this section, we extend the above results to more general billiard models with cusps at flat points. We mainly consider dispersing billiards with one cusp at a flat point.

The first model is a Lorentz gas with finite horizon, see Figure ??  Let $\mathbb{T}^2$ be a unit torus, and $\bB_i$, $i=1,\cdots, l_0$ be a finite number of convex scatterers in $\mathbb{T}^2$, for some $l_0>2$. We assume:\\
\\
(1) There exists $K_{\min}>0$, such that for any point $p\in \partial\bB_i$, the curvature $\cK(p)\geq K_{\min}$, for $i=3,\cdots, l_0$. \\
(2) $\bB_1$ and $\bB_2$ are tangent at a unique flat point $P$, and the cusp at $P$ satisfies assumption (\textbf{h1})-(\textbf{h2}). Moreover, we assume the tangent line of $\partial\bB_1$ is perpendicular to the scatter $\bB_3$.\\

We define the billiard table as $Q_2=\mathbb{T}^2\setminus (\cup_i \bB_i)$. Let $\cF_2$ be the billiard map on the collision space $\cM_2$ associated with $Q_2$. We  define the induced space $M_2$ as in (\ref{defM}), the return time function $R_2:M_2\to \mathbb{N}$, and the induced map $F_2: M_2\to M_2$, such that for any $x\in M_2$, $F_2 x=\cF_2^{R_2(x)} x$.

By taking $K_0$ large, we can check that $\{R_2=n, n>K_0\}=\{M_n, n>K_0\}$ are essentially identical for both systems $(\cF,\cM)$ and $(\cF_2,\cM_2)$. The only differences are those level sets with smaller indices, which play very minor  roles in the study of decay rates of correlations.  Thus one can obtain the same decay rates of correlations.

The second model is the dispersing billiard with corners and a cusp.
Let  $\Gamma_i$, $i=1,\cdots, l_0$ be a finite number of convex curves, for some $l_0>2$. We assume:\\
\\
(1) There exists $K_{\min}>0$, such that for any point $p\in \Gamma_i$, the curvature $\cK(p)\geq K_{\min}$, for $i=3,\cdots, l_0$. \\
(2) $\Gamma_1$ and $\Gamma_2$ are tangent at a unique flat point $P$, and the cusp at $P$ satisfies assumption (\textbf{h1})-(\textbf{h2}). Moreover, we assume the tangent line of $\Gamma_1$ is perpendicular to the boundary $\Gamma_3$.\\

We define the billiard table as $Q_3$, such that $\partial Q_3=\cup_i \Gamma_i$. Let $\cF_3$ be the billiard map on the collision space $\cM_3$ associated with $Q_3$. We  define the induced space $M_3$ as in (\ref{defM}), the return time function $R_3: M_3\to \mathbb{N}$, and the induced map $F_3: M_3\to M_3$, such that for any $x\in M_3$, $F_3 x=\cF_3^{R_3(x)} x$.

Again by taking $K_0$ large, we can check that $\{R_2=n, n>K_0\}=\{R_3>K_0\}$ are indeed identical for both systems $(\cF_2,\cM_2)$ and $(\cF_3,\cM_3)$. The only differences are those level sets with smaller indices, which play very minor  roles in the study of decay rates of correlations.  Thus one can again obtain the same decay rates of correlations.
\section{Appendix: Proof of Proposition \ref{KBbd}.}
Since  the Euclidean norm and the $p$-norm are uniformly equivalent at
the points $x\in M_N$ and $F x$, we can use the $p$-norm in our estimations.
Assume $W\subset M_N$ is an unstable curve.  Then the expansion factor in the p- metric along $x\in W$ satisfies
\beq\label{verticalLa}
  \Lambda(x) = \prod_{n=0}^{N} \bigl(1 + \tau(x_n)
  \cB(x_n)\bigr).\eeq
Let $x_n=\cF^n x=(r_n,\varphi_n)$, for $n=1,\cdots, N$, and $\cK_n=\cK(x_n)$, $\tau_n=\tau(x_n)$.
  Note that  $\cB (x_n)$ satisfies the recursive formula \beq
\label{cBXm}
   \cB(x_{n}) = \frac{2\cK_n}{\cos\varphi_{n}}
   +\frac{1}{\tau_{n-1} + 1/\cB(x_{n-1})}.
\eeq We denote $\lambda_n=\tau_n\cB(x_n)$, and use notation $\sin\gamma_n=\cos\varphi_n$, thus
\beq\label{cBXmtau}
  \lambda_n = \frac{2\tau_n\cK(s_{n})}{\sin\gamma_n}
   +\frac{\tau_n}{\tau_{n-1}}\cdot \frac{1}{1 + 1/\lambda_{n-1}},
\eeq
and
$$\Lambda(x)=\Pi_{n=0}^{N} (1+\lambda_n).$$
By Proposition \ref{prop3} and (\ref{cBXmtau}), to estimate the minimal expansion factor $\Lambda(x)$, we first will  assume \beq\label{4.8}\gamma_1\sim \gamma_N\sim N^{-2/(2+\bar a)}=N^{-\frac{\beta}{2\beta-1}}.\eeq

\begin{proposition} For any $x\in M_N$ satisfying (\ref{4.8}), we have\\
(1) $\lambda_n\sim 1/n$, for $1\leq n\leq N_1$;\\
(2) $\lambda_n\sim 1/N\sim 1/n$, for $N_1< n< N_3$;\\
(3) $\lambda_n\sim 1/(N-n)$, for $N_3<n<N$.
\end{proposition}
Item (2) follows directly from Proposition \ref{prop3}(2). This also implies that the total expansion factor during the interval $[N_1,N_3]$ is of order $\cO(1)$. Thus we can ignore it when calculation the expansion factors. We will prove two lemmas below corresponding to item (1) and item (3), respectively.

Using (\ref{bartaun}), we know that
$$\frac{2\tau_n\cK_n}{\sin\gamma_n}=\frac{D}{n^2}+\cO(\alpha_n^3)$$ where $D=\frac{2\bar a}{(2+\bar a)^2}$,
and by (\ref{ratiotau}), $$\frac{\tau_{n+1}}{\tau_{n}}=1-\frac{B}{n}+\cO(\alpha_n^2),$$ where $B=\frac{4}{2+\bar a}=\frac{2\beta}{2\beta-1}.$

Now we use the relation (\ref{cBXm}) to  get the estimation for $\lambda_n$.\begin{lemma} \label{LmB}
For all $1\leq m < N_1$ we have
\beq \label{cBmain}
   \lambda_{m} \geq
   A\biggl[ m + C_3 \ln m
   + C_4\biggr]^{-1},
\eeq
where $A>0$ satisfies $A^2 + A(B-1) = D$, hence
$A=\frac{\beta-1}{2\beta-1}$, and $C_3,C_4>0$ are sufficiently
large constants. Moreover, $$ \Biggl[\prod_{m=0}^{N_1-1} \bigl(1 + \tau(X_m)
     \cB(X_m)\bigr)\Biggr]
     >C N^{\frac{\beta-1}{2\beta-1}}, $$
     for some constant $C>0$.
\end{lemma}

\begin{proof} We use induction on $m$. For $m=1$ the validity of
(\ref{cBmain}) is guaranteed by choosing $C_4$ large enough.
Assume that (\ref{cBmain}) is valid for some $m <N_1$. Due to
(\ref{cBXmtau})  it is enough to verify
\begin{align*}
\frac{D}{\Bigl[ m + C_1' \ln m +
   + C_2'\Bigr]^{2}}+ \frac{A(1-\frac{B}{m}+\cO(\alpha_m^2))}{ A+ m + C_3 \ln m + C_4}   >
   \frac{A}{ m+1 + C_3 \ln (m+1)
   + C_4},
\end{align*}
provided $C_3,C_4>0$ are large enough. Here
$B=\frac{4}{2+\bar a}$.
It is easy to see that
\begin{align*}
   \frac{A}{ m+1 + C_3 \ln (m+1) +
    C_4}
   -\frac{A(1-\frac{B}{m}+\cO(\alpha_m^2))}{ A + m + C_3 \ln m +
    C_4} < \frac{A^2-A(B-1)}{\Theta},
\end{align*}
where $\Theta$ denotes the product of the two denominators. Thus
it is enough to verify
$$
   \frac{D}{\Bigl[ m + C_1' \ln m + C_2' \Bigr]^{2}} > \frac{A^2+A(B-1)}{\Theta}.
$$
We recall that $A^2 + A(B-1) = D$, since $A=\frac{\bar a}{2+\bar a}$. Thus it is enough to verify
\beq \label{Theta}
   \Theta > \biggl[ m + C_1' \ln m + C_2'\biggr]^{2}.
\eeq
The leading term $m^2$ appears on both sides and cancels out.
Keeping only the largest non-cancelling terms on both sides of
(\ref{Theta}) we obtain
$
  C_3m\ln m
  >
  C_1'm\ln m ,
$
which can be ensured by choosing $C_3$  large enough.
This implies (\ref{Theta}).
Note that due to (\ref{cBXmtau}), we
have
$$
     \ln \Biggl[\prod_{m=0}^{N_1-1} \bigl(1 + \tau(X_m)
     \cB(X_m)\bigr)\Biggr]
     >\sum_{m=1}^{N_1} \biggl[\frac{A}{m} +
     \frac{2C_3\,\ln m}{m^2}\biggr],
$$
with a sufficiently large constant $C_3>0$. Therefore,
$$
     \ln \Biggl[\prod_{m=0}^{N_1-1} \bigl(1 + \tau(X_m)
     \cB(X_m)\bigr)\Biggr]
     >A\ln N_1 + \text{const} > A\ln N +\text{const}.
$$
 Lastly, note
that $A =\frac{\beta-1}{2\beta-1}$, which completes the proof of
the lemma.

\end{proof}
Next we consider the expansion factor for $N_3<n<N$.
\begin{lemma} \label{LmB2} For any $n\in [N_3+1, N-1]$, we denote $m=N-n+1$.
Then we have
\beq \label{cBmain3}
   \lambda_{N-m} \geq
   A\biggl[ m + C_3 \ln m
   + C_4\biggr]^{-1},
\eeq
where $A>0$ satisfies $A^2 - A(B-1) = D$, with $B=\frac{4}{2+\bar a}$, hence
$A=\frac{\beta}{2\beta-1}$, and $C_3,C_4>0$ are sufficiently
large constants. Moreover, $$ \Biggl[\prod_{n=N_3}^{N-1} \bigl(1 + \lambda_{n}\bigr)\Biggr]
     >C N^{\frac{\beta}{2\beta-1}}, $$
     for some constant $C>0$.
\end{lemma}
\begin{proof}
By the time reversibility, we denote $m=N-n+1$.
Assume that (\ref{cBmain3}) is valid for some $m <N_1$. Due to
(\ref{cBXmtau}) and the time reversibility, it is enough to verify
\begin{align*}
\frac{D}{\Bigl[ m + C_1' \ln m
   + C_2'\Bigr]^{2}}+ \frac{A(1+\frac{B}{m}+\cO(\alpha_{N-m}))}{ A+ m + C_3 \ln m +  C_4}   >
   \frac{A}{ m-1 + C_3 \ln (m-1)
   + C_4},
\end{align*}
provided $C_3,C_4>0$ are large enough.
Thus one can check that a sufficient condition for the above inequality is $A^2-A(B-1)=D$, which implies that $A=\frac{\beta}{2(\beta-1)}$, as we claimed.

Note that  we
have
$$
     \ln \Biggl[\prod_{n=N_3}^{N-1} \bigl(1 + \lambda_n\bigr)\Biggr]
     >\sum_{m=1}^{N-N_3} \biggl[\frac{A}{m} +
     \frac{2C_3\,\ln m}{m^2}\biggr],
$$
with a sufficiently large constant $C_3>0$. Therefore,
$$
     \ln \Biggl[\prod_{n=N_3}^{N-1} \bigl(1 + \lambda_n\bigr)\Biggr]
     >A\ln (N-N_3) + \text{const} > A\ln N +\text{const}.
$$
 Lastly, note
that $A =\frac{\beta}{2\beta-1}$, which completes the proof of
the lemma.
\end{proof}
After the last collision, the particle leaves the cusp and
flies back to the boundary $\Gamma_3$. According to (\ref{cBmain3}), during the exit period, $\lambda_{n}\sim {(N-n+1)}^{-1}$, and $$\tau_n\sim \frac{\alpha_n^{\frac{\beta}{\beta-1}}}{(N-n+1)\alpha_n}\sim N^{-\frac{\beta}{(2\beta-1)(\beta-1)}} (N-n+1)^{-\frac{2\beta}{2\beta-1}}.$$ Thus \beq\label{cBN1}\cB(x_n)\sim \lambda_n/\tau_n\sim N^{\frac{\beta}{(2\beta-1)(\beta-1)} } (N-n+1)^{\frac{1}{2\beta-1}}.\eeq
Although the last collision we have $\tau_N\sim 1$, but (\ref{cBN1}) still holds for $n=N$, as it was derived from previous collisions before and at $x_N$, which does not depend on $\tau_N$. Thus the expanding factor contributed by the  collisions at $x_N$ satisfies
$$1+\tau_N\cB(x_N)\sim N^{\frac{\beta}{(2\beta-1)(\beta-1)}}, $$
under assumption (\ref{4.8}).

On the other hand, for the last collision, when $\gamma_N$ fails (\ref{4.8}), we have
\begin{align*}
\frac{|D_{x_N}\cF (d x_N)|_p}{|dx_N|_p}&=1+\tau(x_N)\cB(x_N)=1+\tau(x_N)\left(\cB^-(x_N)+\frac{2\cK(s_N)}{\cos\varphi_N}\right)\\
&\sim  \tau_{\min}\cB^-(x_N)+\tau_{\min}\frac{2\cK(s_N)}{\cos\varphi_N}\sim N^{\frac{\beta}{(2\beta-1)(\beta-1)} }+\frac{\alpha_1^{\frac{\beta-2}{\beta-1}}}{\cos\varphi_N} \\
&\sim N^{\frac{\beta}{(2\beta-1)(\beta-1)} }\left(1+\frac{N^{-\frac{\beta(\beta-3)}{(2\beta-1)(\beta-1)}} } {\cos\varphi_N}\right),
\end{align*}
where we have used the fact that $\cB^-(x_N)>0$, and the free path between $x_N$ and $F x$ is uniformly bounded away from $\tau_{\min}$.

Combining the above facts, we have
\begin{align*}
\Lambda(x)\sim N^{\frac{\beta-1}{2\beta-1}}\cdot N^{\frac{\beta}{2\beta-1}}\cdot N^{\frac{\beta}{(
2\beta-1)(\beta-1)}}=N^{1+\frac{\beta}{(
2\beta-1)(\beta-1)}}.\end{align*}

For points $x\in M_N$, when $\gamma_1$ fails to satisfy (\ref{4.8}), the expression between the first and second collision is
\begin{align*}
\frac{|D_{x_1}\cF (d x_1)|_p}{|dx_1|_p}&=1+\tau(x_1)\cB(x_1)=1+\tau(x_1)\left(\cB^-(x_1)+\frac{2\cK(s_1)}{\cos\varphi_1}\right)\\
&=1+\tau(x_1)\left(\frac{2\cK(s_1)}{\cos\varphi_1}+\frac{1}{\tau(x)+\frac{1}{\cB^-(x)+\frac{2\cK(r)}{\cos\varphi}}}\right)\\
&\geq 1+\tau(x_1)\left(\frac{2\cK(s_1)}{\cos\varphi_1}+\frac{1}{\tau_{\max}+\frac{1}{2\cK_{\min}}}\right)\\
&=1+\frac{2\tau(x_1)\cK(s_1)}{\cos\varphi_1}+\cO(\tau(x_1))\sim 1+\frac{N^\frac{-\beta}{2\beta-1}}{\cos\varphi_1},\end{align*}
where we have used the fact that $\cB^-(x)>0$ for any unstable vector in $\cC^u(x)$, and our estimations on  $\alpha_1$ and $\gamma_1$ in Proposition \ref{prop3}.

\medskip\noindent\textbf{Acknowledgement}.
This paper is written in memory of   Professor Nikolai Chernov. It was him who lead me to the beautiful research field of chaotic billiards.
The author  is also partially supported by NSF (DMS-1151762) and   a grant from the Simons
Foundation (337646, HZ).

\end{document}